\documentclass[11pt]{amsart}

\usepackage{hyperref} 
\usepackage[dvips]{graphicx}
\usepackage[T1]{fontenc}
\usepackage{mathptmx}

\usepackage{color}

\DeclareMathOperator{\trace}{trace}
\DeclareMathOperator{\Ad}{Ad}

\newcommand{\bbar}{\begin{pmatrix}}
\newcommand{\ebar}{\end{pmatrix}}

\newcommand{\bdm}{\begin{displaymath}}
\newcommand{\edm}{\end{displaymath}}
\newcommand{\beq}{\begin{equation}}
\newcommand{\beqa}{\begin{eqnarray}}
\newcommand{\beqas}{\begin{eqnarray*}}
\newcommand{\eeq}{\end{equation}}
\newcommand{\eeqa}{\end{eqnarray}}
\newcommand{\eeqas}{\end{eqnarray*}}
\newcommand{\dd}{\textup{d}}

\newcommand{\C}{{\mathbb C}}

\newcommand{\real}{{\mathbb R}}
\newcommand{\SSS}{{\mathbb S}}

   \newtheorem{theorem}{Theorem}[section]
   \newtheorem{proposition}[theorem]{Proposition}
   
   \newtheorem{lemma}[theorem]{Lemma}
   \newtheorem{definition}[theorem]{Definition}

 \theoremstyle{remark}
   \newtheorem{example}[theorem]{Example}
   \newtheorem{remark}[theorem]{Remark}
   
\numberwithin{equation}{section}

\begin{document}

\title[Pseudospherical surfaces with singularities]{Pseudospherical surfaces with singularities}

\begin{abstract}
We study a  generalization of constant Gauss curvature $-1$ surfaces in Euclidean $3$-space,
based on Lorentzian harmonic maps, that we call pseudospherical frontals.
We analyze the singularities of these surfaces, dividing them into those of characteristic 
and non-characteristic type.   We give methods for constructing all non-degenerate
singularities of both types, as well as many degenerate singularities.
We also give a method for solving the singular geometric Cauchy problem: construct a 
pseudospherical frontal containing a given regular space curve as a non-degenerate singular curve.
The solution is unique for most curves, but for some curves there are infinitely many solutions, and
this is encoded in the curvature and torsion of the curve.

 \end{abstract}

\author{David Brander}
\address{Department of Applied Mathematics and Computer Science\\ Matematiktorvet, Building 303 B\\
Technical University of Denmark\\
DK-2800 Kgs. Lyngby\\ Denmark}
\email{dbra@dtu.dk}



\maketitle


\section{Introduction}
It is a well known theorem of Hilbert that there do not exist complete isometric immersions in $\real^3$ of
surfaces with constant negative Gauss curvature $K=-1$.
These surfaces have nevertheless been much
studied since classical times.  The integrability condition is the sine-Gordon equation $\phi_{xy}= \sin \phi$,
where $x$ and $y$ are unit speed asymptotic coordinates and $\phi$ is the angle between the
asymptotic directions.  Most of the literature on these surfaces deals with them via the solutions
of this equation, naturally leading to singularities along the curves $\phi=n \pi$ for integers $n$.
A more general approach for pseudospherical surfaces
 is the formulation in terms of Lorentz-harmonic maps.  The Gauss map $N$ of a 
pseudospherical surface is harmonic with respect to the Lorentzian metric induced by the second 
fundamental form.    Conversely, if we restrict to \emph{weakly regular} harmonic maps,
i.e. those where the derivatives $N_x$ and $N_y$ with respect to a null coordinate system never vanish,
then these maps correspond to solutions of the sine-Gordon equation. 
The associated surfaces are called weakly regular pseudospherical surfaces, and this has been the
standard class of pseudospherical surfaces investigated in the literature.

In this article we aim to study the natural singularities of pseudospherical surfaces. 
We will drop the weak regularity assumption, as it serves only to
make a connection with the sine-Gordon equation.  This connection is not needed in the harmonic
map approach.  Given a harmonic map $N: S \to \SSS^2$, from
a simply connected Lorentz surface, there is a canonically associated map $f: S \to \real^3$, unique up to a translation,
 that is pseudospherical wherever it is immersed, and such that $\dd f$ is orthogonal to $N$ 
(see Section \ref{harmonicmapssection}).
We take such maps $f$ as the  definition of a generalized pseudospherical surface.

Abandoning the identification with solutions of the sine-Gordon equation is advantageous for two reasons.
In the first place, in order to solve the Cauchy problem along an arbitrary non-characteristic curve,
it is necessary to choose asymptotic coordinates $(x,y)$ such that the curve is given by $y=\pm x$.
This can always be achieved, at least locally, but not if we require that the coordinate lines are 
\emph{constant speed}, the choice for which the angle between the coordinate curves 
is a solution of sine-Gordon.  The second advantage is that we are interested in the natural singularities
of these surfaces, and for many of these (for example the bifurcating cusp lines in figure \ref{viviani_fig},
or the rank zero singularities in Figure \ref{rankzero_figure}), there is no corresponding local solution of
the sine-Gordon equation.

 We will use a variant of the generalized d'Alembert method given by M. Toda \cite{todaagag} to study 
the surfaces.  In brief, a loop group
 lift $\hat F$ of a harmonic map  is obtained, via integration and a loop group decomposition, from a 
\emph{potential pair} $(\hat \chi, \hat \psi)$ of loop algebra valued $1$-forms along a pair
of transverse null-coordinate lines. Essentially, the solution is thus given by more or less arbitrary functions of one variable along two characteristic lines, in analogue with the d'Alembert solution of the wave equation.
 The challenge is to find the potentials that correspond to
particular geometric properties, as the geometry is difficult to see in the potentials.  
To address this problem, in joint work with
 M. Svensson \cite{dbms1}, we defined special potentials that allow one to solve a geometric Cauchy problem: 
 find a surface that contains a given curve with prescribed surface normal.  Here we generalize these
potentials to the case where the curve is required to be a \emph{singular} curve, in place of prescribing the
surface normal.

\subsection{Main results}
A \emph{frontal} is a differentiable map $f$ from a surface $M$ into $\real^3$ that locally has a well-defined unit
normal, that is a map $N$ into $\SSS^2 \subset \real^3$ such that $\dd f$ is orthogonal to $N$.
Generalized pseudospherical surfaces, as defined here, are frontals, and we may thus call them
 \emph{pseudospherical frontals}. If the map $(f,N): M\to \real^3 \times \SSS^2$ is everywhere regular, then $f$ is called a \emph{(wave) front}.    A pseudospherical frontal is a wave front if and only if it is weakly regular.   That is, wave front solutions are exactly those that 
correspond to solutions of the sine-Gordon equation.

A point $p$ on a frontal $f$ is called a \emph{singular point} if the derivative $\dd f$ has rank less than $2$ at
$p$, and the local singular locus is called a \emph{singular curve}. The singular point $p$ is
\emph{non-degenerate} if the singular curve is locally a regular curve in $M$.  The image  in $\real^3$ of a non-degenerate singular curve need not be a regular curve,
demonstrated by the case of a swallowtail singularity 
or a cone singularity (Figure \ref{singularitiesfig}).
  Below we will divide non-degenerate singular curves into two types, 
\emph{characteristic} singular curves that are always tangent to a null coordinate direction, and
\emph{non-characteristic}, those that are never tangent to a null direction.

Theorem \ref{mainthm1} gives the potentials for constructing 
all non-degenerate non-characteristic singular curves, together with the conditions
on the data for cuspidal edges, swallowtails and cone singularities.  We then use this to prove 
Theorem \ref{mainthm2}, which states that given an arbitrary space curve with non-vanishing
curvature $\kappa$, and torsion $\tau \neq \pm 1$, there is a \emph{unique} pseudospherical wave front that
contains this curve as a cuspidal edge.  Moreover, the potentials are given by a very
simple formula in terms of $\kappa$ and $\tau$. We use this formula to compute several 
examples. In fact the potentials in Theorem \ref{mainthm1} generate a pseudospherical frontal
from an arbitrary pair of functions $\kappa$ and $\tau$. At a point where $\kappa$ vanishes,
the singular curve is degenerate. At a point where $|\tau|=1$ the surface is a frontal but not a wave front,
and the singular curve is also degenerate.
Examples are shown in Figures  \ref{viviani_fig}, \ref{inflectionfig1} and  \ref{figureunbounded}.

\begin{figure}[ht]
\centering
$
\begin{array}{cc}
\includegraphics[height=45mm]{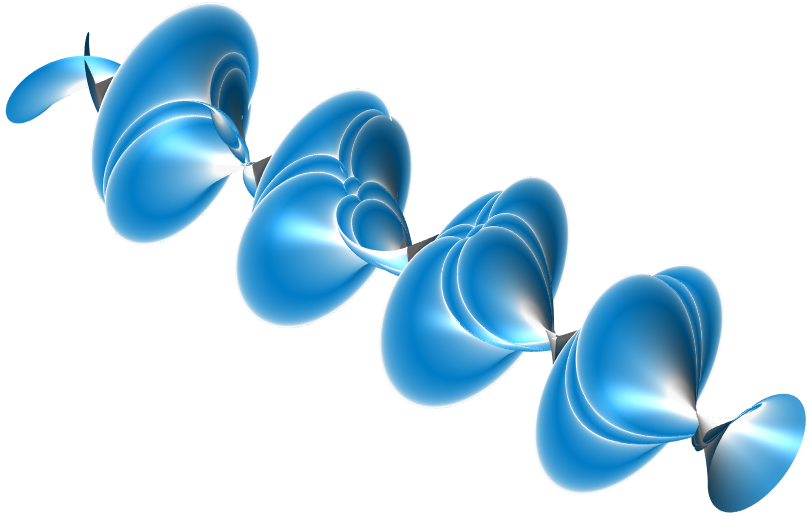}   \quad & \quad 
\includegraphics[height=42mm]{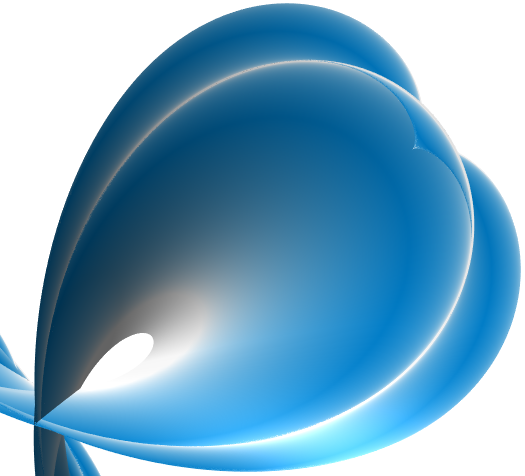}  
\end{array}
$
\caption{The pseudospherical surface
 generated by a Viviani figure 8 space curve. The curve has non-vanishing curvature, but
$|\tau|=1$ at four points. This surface is a frontal but not a wave front. (Example \ref{degenexamples}).  }
\label{viviani_fig}
\end{figure}

In Section \ref{charsection} we analyze the problem for \emph{characteristic} singular curves.
These singularities are non-generic, but nevertheless of some interest.
For example, a \emph{weakly regular} pseudospherical surface (i.e. a wave front) contains a non-degenerate
characteristic singular curve if and only if this curve is a straight line segment.  
For a general frontal, the singular curve, if it is not a straight line, must instead have non-vanishing
curvature and constant torsion $\tau = \pm 1$, incidentally the same conditions that are satisfied by
asymptotic curves on a \emph{regular} pseudospherical surface.  In the characteristic case, the solution is not unique,
and there are infinitely many pseudospherical frontals containing a prescribed curve of
the allowed type.  We give the precise statement and
 the potentials for all solutions in Theorem \ref{mainthm3}.  

\begin{figure}[ht]
\centering
$
\begin{array}{cc}
\includegraphics[height=25mm]{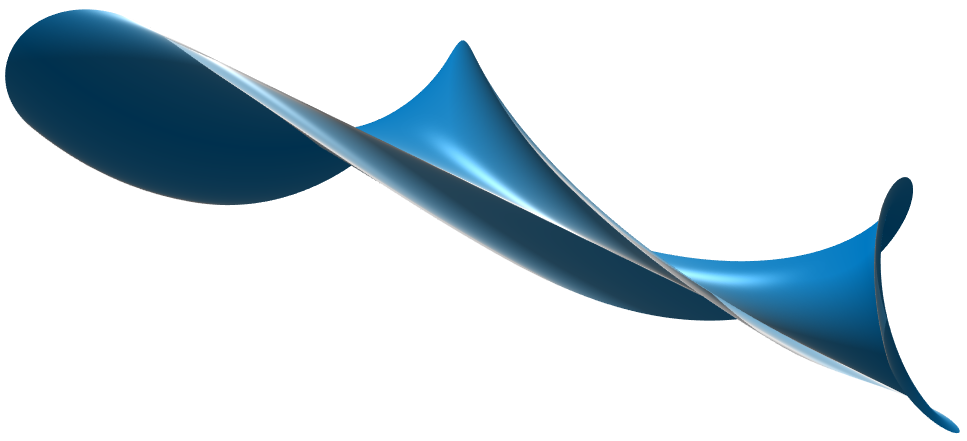} \quad   & \quad 
\includegraphics[height=25mm]{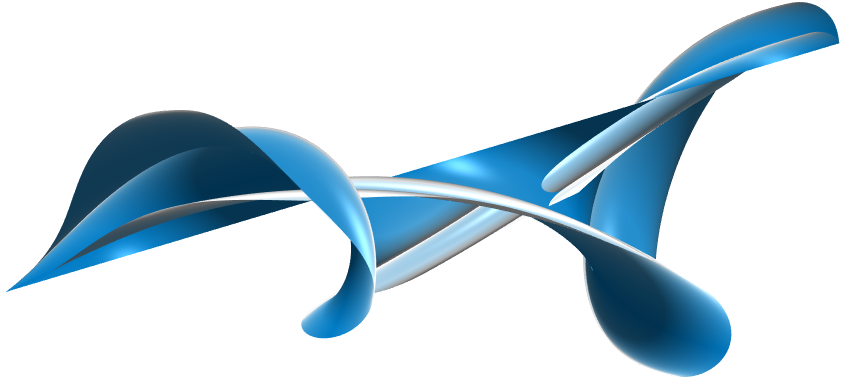}
\end{array}
$
\caption{Left: one of many pseudospherical fronts that contain a straight line as a singular curve: 
Theorem \ref{mainthm3}, with  $\kappa=0$, $\alpha=1$, $\beta(t)=t$. Right:
Example \ref{nonfrontlinesex},  a higher  order ``cuspidal
edge'', $\kappa=0$, $\alpha(t)=t^2$, $\beta(t)=t$. This surface is not a wave front. } \label{linefig}
\end{figure}
We have computed many examples of solutions using a numerical implementation of the generalized
d'Alembert method.\footnote{Currently available at 
\href{http://davidbrander.org/software.html}{http://davidbrander.org/software.html}}
We have tried to include some representative images throughout the article,  as well as further
examples illustrating degenerate singularities and surfaces generated from curves with unbounded
curvature in Section \ref{numericssection}.  The surfaces are colored here according to mean curvature,
which generally blows up near singularities, showing the singular curves more clearly.

\subsection{Concluding remarks}
This work is part of a series  investigating how to analyze the singularities arising naturally in 
integrable systems formulations of geometric problems \cite{brs, sbjorling, dbms2, spherical}.
The singularities in each case studied arise in a different way.  For spacelike and timelike constant mean curvature surfaces
in Minkowski $3$-space \cite{brs, sbjorling,dbms2},  singularities are caused by
the break down of the Iwasawa  and Birkhoff loop group decompositions
for non-compact groups.  Approaching such points, the direction of the surface normal becomes null, and so the (harmonic) unit
normal is not defined.  For constant Gauss curvature surfaces in $\real^3$ the loop group decompositions are globally defined,
and hence the unit normal is well-defined everywhere, but this does not guarantee that
the surface is regular.  This is because the unit normal is harmonic with respect to the metric induced by the \emph{second}
fundamental form, so the existence of conformal coordinates with respect to this metric does not imply 
surface regularity.   Positive curvature surfaces, studied in \cite{spherical},  differ substantially in treatment from 
negative curvature surfaces, because the former constitute an elliptic problem corresponding to Riemannian harmonic maps,
and the latter case, treated here, is hyperbolic and corresponds to Lorentzian harmonic maps. 

We generally consider maps to be in the smooth category.  The methods we use 
involve only integration and loop group decompositions, which preserve smoothness: if 
real analytic data are given, then the solutions are also real analytic. Our solutions, as frontals,
are defined globally, because the Birkhoff decomposition used is shown in \cite{jgp} to be global.
We work with a simply connected (which implies contractible) Lorentz surface $S$.
For non-trivial topologies this amounts to working on the universal cover. Note, however,
that by Kulkarni's theorem \cite{kulkarni}, there are infinitely many Lorentzian conformal structures
on the plane, and not all of these can be realized as conformal submanifolds of the Lorentz plane $\real^{1,1}$.
This raises interesting questions for the global theory of pseudospherical frontals. 

Andrey Popov  \cite{popov} proved the existence and uniqueness part of our Theorem \ref{mainthm2},
by using  the sine-Gordon equation.  The potentials given in Theorem \ref{mainthm2} improve
 this result by including solutions for
curves where  $\kappa$ vanishes or $|\tau|$ takes the value $1$, and by providing 
a means of easily computing the solutions. 
   Popov concluded that a pseudospherical
surface is uniquely determined by a cuspidal edge on its boundary, but this is not strictly accurate: even if we 
restrict to the class of pseudospherical wave fronts (as he did), there exist cuspidal edges (necessarily straight lines) that 
are characteristic curves. For such a curve, there are infinitely many different pseudospherical wave fronts 
that contain it as a cuspidal edge.

An important motivation for studying the singularities of pseudospherical surfaces is to characterize
the natural boundaries of the regular surfaces, given that there are no complete immersions.
See, e.g., \cite{amsler, wissler}.
Generalizations that \emph{include} the singular curves as a part of the surface have previously been studied
 within the framework of weakly regular surfaces. In this article, we construct real analytic pseudospherical frontals 
(Examples \ref{nonfrontlinesex}  and \ref{deg_spiralex} ) that are
immersed on open dense sets, but have non-degenerate
singular curves where the surface is not weakly regular. This demonstrates that the weakly regular framework is not
sufficiently general for the task of including even \emph{regular} boundary curves of immersed pseudospherical surfaces.  Given this, and 
the direct relationship
between arbitrary Lorentz harmonic maps and globally defined pseudospherical frontals,
we conclude that frontals are a more natural candidate 
for a global theory of pseudospherical surfaces.

\section{Generalized pseudospherical surfaces}  
We first summarize necessary background material on pseudospherical surfaces and the loop group representation.
For more references, see, for example, \cite{bobenko1994, melkosterling, pinkall2008}.
\subsection{Lorentz surfaces and box charts}
Any pseudospherical immersion has a natural Lorentz structure induced by the second 
fundamental form.  We therefore outline a little background on Lorentz surfaces
from Weinstein \cite{weinstein}.

A Lorentz surface $(S,[h])$ is an oriented $\C^\infty$ surface $S$ equipped with a conformal 
equivalence class of indefinite metrics $[h]$.  There is naturally associated an ordered
pair of nowhere parallel null direction fields $\mathcal{X}$ and $\mathcal{Y}$.  A local 
\emph{proper null coordinate system} with respect to $[h]$ is a local coordinate 
chart $(x,y)$ such that $\partial_x$ and $\partial_y$ are parallel to $\mathcal{X}$ and 
$\mathcal{Y}$ respectively and $h=2 B \dd x \dd y$ for some positive function $B$.

The Lorentzian analogue to a holomorphic chart of a Riemann surface is a box chart.
A pair of charts $\phi = (x,y)$, $\hat \phi = (\hat x, \hat y)$, on a surface $S$ are \emph{$C^\square$-related} if the orientation, and the directions $\partial_x$ and $\partial_y$
are preserved by the transition function, that is $\hat \phi \circ \phi^{-1}(x,y) = (f(x),g(y))$
with $f^\prime g^\prime >0$.  A \emph{$C^\square$-atlas} $\mathcal{A}^\square$ is a 
subatlas of the atlas of $S$ in which all charts are $C^\square$-related. A \emph{box surface}
is an ordered pair $(S, \mathcal{A}^\square)$, consisting of a surface and a maximal 
$C^\square$-atlas, and any element of $\mathcal{A}^\square$ is called a \emph{box chart}.  

By Theorem 1 of \cite{weinstein}, box surfaces are in one-one correspondence
with Lorentz surfaces $(S, [h])$, where $[h]$ is a conformal equivalence class of Lorentz
metrics.  In particular, given a Lorentz surface $(S,h)$, the set of all proper null coordinate
charts is a maximal $C^\square$ atlas on $S$. 

A \emph{grid box} in $\real^2$ is a product of intervals $B=(a,b)\times(c,d)$ where
$-\infty \leq a < b\leq \infty$ and $-\infty \leq c < d\leq \infty$.  Since the property
of being a grid box is  preserved by the transition functions of $C^\square$-related charts,
the concept of a grid box is well defined on a Lorentz surface. We call $\phi^{-1}(B)$
a grid box on $S$ if $B$ is a grid box and $\phi$ is a box chart.

\subsection{Lorentz harmonic maps and the associated pseudospherical frontal}   \label{harmonicmapssection}
Let $(S,h)$ be a simply connected Lorentz surface. 
Suppose $N: S \to \SSS^2$ to be a smooth  map.  
Then $N$ is harmonic if and only if the mixed partial derivative
$N_{xy}$ is proportional to $N$, otherwise stated as $N \times N_{xy}=0$, where $(x,y)$
are any null coordinate system (box chart).

Consider now the system 
\beq  \label{assocfrontaleqn}
f_x = N \times N_x , \quad
f_y = - N \times N_y,
\eeq
for a map $f: S \to \real^3$.
The compatibility of the system \eqref{assocfrontaleqn}, i.e.
$\partial _y ( N \times N_x ) = \partial _x (- N \times N_y)$
 is equivalent to the equation $N \times N_{xy}=0$, i.e. to the harmonicity of $N$. 
\begin{definition}
The smooth map $f: S \to \real^3$, unique up to a translation, obtained by integrating the system 
\eqref{assocfrontaleqn}  is called the \emph{pseudospherical frontal associated to $N$}.
The map $L=(f,N): S \to \real^3 \times \SSS^2$ is called the \emph{Legendrian lift of $f$}.
\end{definition}
Clearly $\dd f$ is  orthogonal to $N$, and so $f$ is a frontal. 
At points where $f$ is an immersion, the Gauss curvature is $-1$, and the null coordinates 
are \emph{asymptotic coordinates} for $f$ (see below). Hence the name pseudospherical frontal.   

 Conversely, if $\tilde f: S \to \real^3$ is a regular constant Gauss curvature $-1$ surface, where $S$ is
simply connected, it is well known that one can find a global asymptotic coordinate system for $\tilde f$,
and that the unit normal is a harmonic map with respect to the Lorentz structure defined by the 
second fundamental form.  Hence all standard pseudospherical surfaces are obtained in the 
above manner from their Gauss maps.

\subsection{The extended frame}  \label{extendedframesect}
Let $K$ denote the diagonal subgroup of $SU(2)$, and represent $\SSS^2$ as the symmetric
space $SU(2)/K$, with projection $\pi: SU(2) \to \SSS^2$ given by $\pi(g) = \Ad_{g}e_3$,
where
\bdm
e_1=\frac{1}{2}\bbar 0 & i \\ i & 0 \ebar, \quad
e_2=\frac{1}{2}\bbar 0 & -1 \\ 1& 0 \ebar, \quad
e_3=\frac{1}{2}\bbar i & 0 \\ 0 & -i \ebar, 
\edm
are an orthonormal basis for $\mathfrak{su}(2)$, with respect to the inner product
$\langle X, Y \rangle = -2\trace(XY)$.  We have the commutators
$[e_1,e_2]=e_3$, $[e_2,e_3]=e_1$ and $[e_3,e_1]=e_2$, so that the cross-product in
$\real^3 =\mathfrak{su}(2)$ is 
$$
A \times B = [A,B].
$$
Let $N: S \to \SSS^2 = SU(2)/K$ be a harmonic map, as above,
and $F: S \to SU(2)$ any lift of $N$, i.e. a map such that 
$N=\pi(F) = \Ad_F e_3$.
We can express the Maurer-Cartan form of $F$ as
\bdm
\alpha:= F^{-1} \dd F = (U_{\mathfrak{k}} + U_\mathfrak{p}) \dd x + (V _{\mathfrak{k}} + V_{\mathfrak{p}})\dd y,
\edm
where the $\mathfrak{k}$ and $\mathfrak{p}$ components are with respect to the 
Lie algebra decomposition $\mathfrak{k} = \textup{span}\{e_3\}$,
$\mathfrak{p} = \textup{span}\{e_1,e_2\}$.  

The equations \eqref{assocfrontaleqn} for the associated pseudospherical frontal can be written
\[  
f_x = \Ad_F U_\mathfrak{p}, \quad
f_y = - \Ad_ F V_\mathfrak{p},
\]
and  $f$ is immersed precisely at the points where $U_\mathfrak{p}$ and $V_\mathfrak{p}$ are
linearly independent.  At such a point, the first and second fundamental forms are
\[
{\mathbb I} = \bbar |U_\mathfrak{p}|^2  & | U_\mathfrak{p} | |V_\mathfrak{p} | \cos \phi \\
      | U_\mathfrak{p} | |V_\mathfrak{p} | \cos \phi  &   |V_\mathfrak{p}|^2  \ebar, \quad
{\mathbb I \mathbb I} = \bbar 0 &  |U_\mathfrak{p}| |V_\mathfrak{p}| \sin \phi  \\
	       |U_\mathfrak{p}| |V_\mathfrak{p}| \sin \phi  & 0 \ebar,
	\]
	where $\phi$ is the angle from $U_\mathfrak{p}$ to $-V_\mathfrak{p}$, and $|\cdot|$ is the standard norm
in $\real^3 \equiv \mathfrak{su}(2)$.
Thus $x$ and $y$ are \emph{asymptotic coordinates} for $f$, and the Gauss curvature is $-1$.

To characterize the harmonicity of $N$ in terms of $F$, we differentiate $N=\Ad_F e_3$ to obtain
\[
\Ad_{F^{-1}}   N_{xy}  =  [U_\mathfrak{p}, [V_\mathfrak{p}, e_3]]  + \left[ \frac{\partial V_\mathfrak{p}}{\partial x} + 
		      [U_\mathfrak{k}, V_\mathfrak{p}] , e_3 \right].
\]
Hence, $N_{xy}$ is proportional to $\Ad_F e_3$ if and only if the $\mathfrak{p}$ part of the right hand side
vanishes, i.e. if and only if  
$\left[ \partial_x V_\mathfrak{p}+ [U_\mathfrak{k}, V_\mathfrak{p}] , e_3 \right] = 0$, and this holds
if and only if
\beq \label{harmoniceqn}
 \partial_x V_\mathfrak{p}+ [U_\mathfrak{k}, V_\mathfrak{p}] = 0.
\eeq
If $\alpha$ is the Maurer-Cartan form of a frame $F$ for an arbitrary smooth map $N: S \to \SSS^2$, we 
can define
\bdm
\alpha_\lambda:= (U_{\mathfrak{k}} + U_\mathfrak{p}\lambda) \dd x + (V _{\mathfrak{k}} + V_{\mathfrak{p}}\lambda^{-1})\dd y,
\edm
where the parameter $\lambda$ takes values in $\C^* :=\C \setminus \{0\}$.
The basis of the loop group setup is that  the Maurer-Cartan equation
\beq \label{mceqn}
\dd \alpha_\lambda + \alpha_\lambda \wedge \alpha_\lambda=0, 
\eeq
is satisfied for all $\lambda$ if and only if Equation \eqref{harmoniceqn} holds, if and only if
$N$ is harmonic.

Fix some point $p \in S$ with  $F(p)=F_0$.  We want to retain the \emph{twisted} structure
	that $\alpha_\lambda$ already has, namely that diagonal and off-diagonal matrix components are 
	respectively even and odd functions of $\lambda$.  We therefore set
\bdm
F_0^\lambda := \bbar a & \lambda b\\ -\bar b\lambda^{-1} & \bar a\ebar, \quad \hbox{where} \quad
 F_0 = \bbar a & b\\ -\bar b& \bar a\ebar.
\edm
 Give that $N$ is harmonic, the Maurer-Cartan equation (\ref{mceqn}) means that,
for any value of $\lambda$, we can solve the equations
\bdm
(F^\lambda)^{-1} \dd F^\lambda = \alpha_\lambda,  \quad F^\lambda(p) =  F^\lambda_0,
\edm
to obtain a family of maps $F^\lambda: S \to SL(2,\C)$, which take values in
$SU(2)$ for real values of $\lambda$, and we have an associated family 
$N^\lambda: S \to \SSS^2$ of harmonic maps given by
\bdm
N^\lambda:= \Ad_{F^\lambda} e_3, \quad \textup{for } \lambda \in \real^*.
\edm
Given a fixed basepoint $p$, the family $N^\lambda$ is independent of the choice of lift $F$ of $N$.
Any other lift is of the form $\tilde F = F D$ where $D$ is a diagonal matrix valued function,
and the extended frame works out to be $\tilde F^\lambda = F^\lambda D$, leaving
 $N^\lambda = \Ad_{F^\lambda}e_3$ unchanged.
Let us call the 
family $N^\lambda$ the \emph{extended harmonic map}, or \emph{the extended unit normal}, and
$F^\lambda$ an extended frame. 
There is a convenient way to obtain the associated pseudospherical  frontal $f$ from  $F^\lambda$.
The \emph{Sym formula} is defined as:
\beq \label{sym}
\mathcal{S}_\lambda (N^\lambda) := \lambda \frac{\partial F^\lambda}{\partial \lambda} (F^\lambda)^{-1}.
\eeq
This formula is independent of the choice of extended frame $F^\lambda$, (given a fixed basepoint), and
hence well defined on $N^\lambda$.   By computing the derivatives one verifies:
\begin{lemma}  \label{pssurfacelemma}
For each $\lambda \in \real^*$, the map
$f^\lambda: S \to \real^3 = \mathfrak{su}(2)$, given by the \emph{Sym} formula:
$f^\lambda = \mathcal{S}_\lambda (N^\lambda)$
is, (up to a translation) the unique pseudospherical frontal associated to the harmonic map $N^\lambda$.
\end{lemma}
The Sym formula was given by A. Sym \cite{sym1985}. A geometric explanation of this formula can be found in \cite{bik}.

Finally, we remark that the choice of basepoint in the construction of the extended harmonic map $N^\lambda$
has no geometric significance.  Choosing a different basepoint will result in a translation of the surface obtained from
the formula $f= \mathcal{S}_1(N^\lambda)$, and this is the same freedom we have in the definition of the 
associated pseudospherical frontal.

\section{Singularities of pseudospherical frontals} \label{singsection}
For notational convenience,
 we now use $\hat X$ instead of $X^\lambda$ to denote a family
of objects parametrized by $\lambda$. For such an object, we also write $X$ for
$\hat X |_{\lambda =1}$.

Analysis of singularities is local, and so, in this section, we are generally discussing
a harmonic map $N:R \to \SSS^2$, where  $R$ is a grid box $I_x \times I_y \subset \real^2$, a product of open intervals. 
A harmonic map $N$ is called \emph{weakly regular} if
the kernel of $\dd N$ is everywhere of dimension at most 1, and never contains a non-zero null vector.
\begin{definition}  \label{admissibledef}
 An \emph{admissible connection} is an integrable family of $1$-forms
$$
\hat \alpha:= (U_{\mathfrak{k}} + U_\mathfrak{p}\lambda) \dd x + (V _{\mathfrak{k}} + V_{\mathfrak{p}}\lambda^{-1})\dd y,
$$
on $R:=I_x \times I_y$,
where $U_{\mathfrak{k}}$, $V_{\mathfrak{k}}$ and  $U_\mathfrak{p}$, $V_\mathfrak{p}$ 
take values respectively in $\mathfrak{k}$ and $\mathfrak{p}$ in $\mathfrak{su}(2)$.
The connection is \emph{weakly regular} at $p \in R$,
 if both $U_\mathfrak{p}$ and $V_\mathfrak{p}$ 
are non-zero at $p$, and \emph{regular} if $U_\mathfrak{p}$ are $V_\mathfrak{p}$ 
are linearly independent at $p$.  The connection is weakly regular or regular if these
conditions hold on the whole of $R$.  An \emph{admissible frame} is a family of maps
$\hat F: R \to SU(2)$ such that $\hat F^{-1} \dd \hat F$ is an admissible connection.
\end{definition}
The problem of constructing harmonic maps $R \to \SSS^2$ is essentially equivalent to that of finding admissible connections. The only freedom in the choice of admissible
frame $\hat F$ is a gauge $\hat F \mapsto \hat F D$, where $D$ takes values in the diagonal subgroup $K \subset SU(2)$.
Equivalently, $\hat \alpha \mapsto  D^{-1} \hat \alpha D + D^{-1} \dd D$.
The harmonic map $N = \Ad_F e_3$ is (weakly) regular if and only if the admissible connection is.

\begin{lemma}  \label{frontlemma}
Let $\hat F$ be an admissible frame, with associated harmonic map $N= \Ad_F e_3$ and 
$f=\mathcal{S}_1(\hat F)$.
 The connection $\hat \alpha := \hat F^{-1} \dd \hat F$ is weakly regular if and only if $f$ is a wave front.
\end{lemma}
\begin{proof}
We have
\[
\Ad_{F^{-1}} \dd f=  U_{\mathfrak{p}} \dd x - V_{\mathfrak{p}} \dd y.
\]
If $\hat \alpha$ is not weakly regular then at least one of $U_{\mathfrak{p}}$
and $V_{\mathfrak{p}} $ is zero at some point.
 Since the derivatives $\dd N$ and $\dd f$ are
computed in terms of these, the rank of $\dd L=(\dd f, \dd N)$ is at most $1$
at this point and
$f$ is not a wave front.

Now suppose that $\hat \alpha$ is weakly regular. We need to show that
$\dd L = (\dd f, \dd N)$ has rank $2$.
Define $W: R \to \SSS^1 \subset \mathfrak{p}$ by
 $W= U_\mathfrak{p}/|U_\mathfrak{p}|$.   We can write
\bdm
U_\mathfrak{p} = A W, \quad \quad V_\mathfrak{p} = -B R_\phi W,
\edm
where $A$ and $B$ are smooth positive real valued functions, $\phi$ is
smooth and real-valued, and $R_\phi$ denotes the rotation of angle $\phi$ in 
the $e_1 e_2$ plane.  The connection is regular when $\phi$ is not an integer 
multiple of $\pi$. 
Writing $W=R_\gamma e_1$, let us multiply the extended frame $\hat F$ on
the right by $D=\textup{diag}(e^{i\gamma/2}, e^{-i \gamma/2})$. This has
no effect on the harmonic map $N=\Ad_F e_3$ or the map $f=\mathcal{S}_1(\hat F)$.
Thus it is equivalent to consider the admissible connection 
$D^{-1} \hat \alpha D + D^{-1} \dd D$, which we now denote by $\hat \alpha$.
The conclusion is that we can assume that 
$$
U_\mathfrak{p}= A e_1, \quad V_\mathfrak{p} = -B (\cos \phi e_1 + \sin \phi e_2).
$$
Now 
\beqas
\Ad_{F^{-1}} \dd N &=& [A e_1 \dd x - B (\cos \phi e_1 + \sin \phi e_2) \dd y \, , \, e_3]\\
  &=& -B \sin \phi \dd y \, e_1 +(-A \dd x + B \cos \phi \dd y) \, e_2,
\eeqas
and 
\bdm
\Ad_{F^{-1}} \dd f = (A \dd x + B \cos \phi \dd y) e_1 + B \sin \phi \dd y \, e_2.
\edm
Since $A$ and $B$ are non-vanishing, it follows that $\dd L = (\dd f, \dd N)$ has rank $2$ and $f$ is a wave front.
\end{proof}

\subsection{The singular curve for pseudospherical wave fronts}  \label{noncharsingularities}
Assume that $\hat \alpha$, $N$ and $f$ are as above, and $\hat \alpha$ is
weakly regular. Using the same choices as in the previous lemma, we have
\beq \label{framedef}
f_x = A \Ad_F  e_1, \quad f_y = B \Ad_F (\cos \phi e_1 + \sin \phi e_2), \quad N = \Ad_F (e_1 \times e_2).
\eeq
Thus 
\beq
f_x \times f_y = A B  \sin \phi N.
\eeq
Since $A$ and $B$ are assumed non-vanishing, the singular set is the set of points
$\sin \phi = 0$, i.e. $\phi = k \pi$, for $k \in {\mathbb Z}$.
A singular point $q$ on a frontal is non-degenerate if and only if one can write
$f_x \times f_y = \mu N$, where $\mu(q)=0$ and $\dd \mu|_q \neq 0$. 
Here we have $\mu = A B \sin \phi$ and $\dd \mu = \pm A B\dd \phi$. Thus the
non-degeneracy condition in our case is
\beq \label{nondegeneracy}
\dd \phi \neq 0.
\eeq
In a neighbourhood of a non-degenerate singular point, the singular set is a 
regular curve in the coordinate domain, and  there is a well-defined $1$-dimensional direction
field $\eta$ along the curve called the \emph{null direction} (not to be 
confused with null coordinate directions!) such that 
$$
\dd f (\eta)=0.
$$
The generic singularities of pseudospherical surfaces were studied by Ishikawa and Machida \cite{ishimach},
 and shown to be cuspidal edges and swallowtails.  For general wave fronts, these singularities can
be identified by the following characterization:
\begin{proposition} \cite{krsuy}
Let $f$ be a wave front and $q$ a non-degenerate singular point. Let
$\sigma(t)$ be a local parametrization for the singular curve around $q$,
with $\sigma(0)=q$.  Then the image of $f$ in a neighbourhood of $q$ is diffeomorphic to:
\begin{enumerate}
\item A cuspidal
edge if and only if $\eta(0)$ is not proportional to $\sigma^\prime(0)$;
\item 
A swallowtail if and only if $\eta(0)$ is proportional to $\sigma^\prime(0)$, and
\bdm
\frac{\dd}{\dd t}(\det(\sigma^\prime(t), \eta(t))\big|_{t=0} \neq 0.
\edm
\end{enumerate}
\end{proposition}
In our situation, assuming, for concreteness' sake that the singular curve is given
locally by $\phi(x,y)=0$, we have $\dd f = (A \dd x + B \dd y)\Ad_F e_1$, and so
the null direction is given on this curve by
\bdm
\eta = B \partial_x - A \partial_y.
\edm
Assume first that the singular curve is not tangent to either $\partial_x$ or
$\partial_y$. In that case we can, after a change of box coordinates (see, e.g., \cite{dbms1}),
  assume that our singular curve is locally given by $y=\varepsilon x$, where $\varepsilon=\pm1$.
	Note that this special choice of coordinates means that we  cannot assume that $A$ and $B$ are
	constant.    Now we have, in the basis
$\partial_x$, $\partial_y$,
\beqas
\eta(t) = (B(t), -A(t)), \quad  \sigma^\prime(t)=(1,\varepsilon),
\quad \frac{\dd}{\dd t}\det(\sigma^\prime(t), \eta(t))= A^\prime(t)+\varepsilon B^\prime(t).
\eeqas
Let us add here that the special case that $A(t)+\varepsilon B(t)\equiv 0$
corresponds to a \emph{cone singularity}, i.e. a non-degenerate singular curve that maps to
a single point. This follows from the formula $\dd f(\sigma^\prime(t)) = (A(t)+\varepsilon B(t))\Ad_F e_1$.
Constructing pseudospherical wave fronts with cone singularities is discussed by Pinkall \cite{pinkall2008}.

Now consider the case that the singular curve is tangent, at a point $p$, to one of the coordinate
directions $\partial_x$ or $\partial_y$.  Then it is not proportional to $\eta$, because
both $B$ and $A$ are non-zero. In this case, by the proposition above,
 the surface is a cuspidal edge at $p$.
We summarize this as:
\begin{theorem}   \label{genericsingthm}
Let $f$ be a pseudospherical wave front.
Suppose that $q$ is a non-degenerate singular point. If the singular curve is
tangent at $q$ to a null coordinate direction then the surface is locally diffeomorphic to a
cuspidal edge at $q$.   Otherwise, there exist box coordinates 
$(x,y)$ such that, in a neighbourhood of $q=(0,0)$, the singular
set is  parametrized by $(x(t),y(t))= (t, \varepsilon t)$, and the image of 
$f$ is diffeomorphic to:
\begin{enumerate}
\item A cuspidal edge if $A(0) + \varepsilon B(0) \neq 0$;
\item A swallowtail if $A(0) + \varepsilon B(0) =0$ and 
$A^\prime (0)+\varepsilon  B^\prime (0) \neq 0$.
\item A cone singularity if  $A(t) + \varepsilon B(t) \equiv 0$,
\end{enumerate}
where $A(t)=|f_x(t, \varepsilon t)|$ and $B(t)=|f_y(t, \varepsilon t)|$.
\end{theorem}

\subsection{Singular curves that are not wave fronts}  \label{charsingsection}
Let us now consider the case that $\hat \alpha$ is \emph{semi-regular} -- meaning
that the derivative of the associated harmonic map $N$ has rank at least $1$ --
but not weakly regular.
This means that at least one of $U_\mathfrak{p}$ and $V_\mathfrak{p}$ is non-zero,
but the other may vanish.  We assume then that  $U_\mathfrak{p} \neq 0$,
the other case being analogous.  We can, as before, assume that
 $U_\mathfrak{p}=A e_1$. After a change of box coordinates, we can take $A=1$.
The angle $\phi$ is not well defined at points where $V_\mathfrak{p}$ vanishes,
so we now have: 
\beqas
U_\mathfrak{k} = u_0 e_3, \quad U_\mathfrak{p}=  e_1,\\
V_\mathfrak{k} = v_0 e_3, \quad V_\mathfrak{p} = ae_1 + be_2,
\eeqas
where $u_0$, $v_0$, $a$ and $b$ are real-valued functions. The integrability
condition $\dd \hat \alpha + \hat \alpha \wedge \hat \alpha=0$ is equivalent to the following set of equations
\bdm
\frac{\partial u_0}{\partial y}= b, \quad \frac{\partial a }{\partial x}=u_0 b, \quad
\frac{\partial b }{\partial x}=-u_0 a, \quad v_0=0.
\edm
Now we have
\bdm
f_x =  \Ad_F e_1, \quad f_y = \Ad_F(ae_1 + be_2), \quad
 f_x \times f_y =  b N.
\edm
Thus the frontal $f$ has a singular point precisely when $b$ vanishes,
i.e. the singular set is given by
$$
b=0,
$$
 and the non-degeneracy condition is $\dd b \neq 0$.
If $a$ is non-vanishing then we are at a weakly regular point, already discussed.
 We therefore consider now a point $q$ at which
\bdm
a(q)=0, \quad b(q)=0, \quad \dd b |_q \neq 0.
\edm
We relabel coordinates so that $q=(0,0)$. 
The integrability conditions above for $a$ and $b$ give, along the line $y=0$,
the system:
\bdm
\quad \frac{\partial a(x,0) }{\partial x}=u_0(x,0)  b(x,0), \quad
\frac{\partial b(x,0) }{\partial x}=-u_0(x,0) a(x,0), \quad a(0,0)=b(0,0)=0,
\edm
which has the unique local solution 
$$
a(x,0)=b(x,0)=0.
$$
Hence, assuming the non-degeneracy condition, which is now 
\[
\partial_y b|_{(x,0)} \neq 0,
\]
 the singular curve is locally given by 
$$
y=0.
$$
The other integrability condition becomes $\partial_y u_0=0$ along 
$y=0$.
The null direction is $\eta = \partial_y$, which is transverse
to the singular curve, but the singularity is not a standard cuspidal edge
because the surface is not a wave front along this curve.  We  call such a singularity a
\emph{higher order cuspidal edge}, because it is non-degenerate and the image of
the singular curve is a regular curve in $\real^3$. A \emph{fold} singularity is of this type.

We have shown that if a pseudospherical surface has a non-degenerate singularity at a point
where the surface is not a wave front, then the singular curve at that point is a
\emph{characteristic curve}, or null coordinate curve.  However, we saw in the previous
section that it is also possible for a weakly regular singular curve to be tangent to 
a characteristic direction.

\section{Prescribed non-characteristic singular curves}
\subsection{The generalized d'Alembert method}   \label{gdamethod}
A well known method for producing essentially all admissible frames
is the \emph{generalized d'Alembert representation} given by M. Toda in \cite{todaagag}.
Here is a summary, using definitions and notation as in \cite{dbms1}:
let $\mathcal{G}:= \Lambda SL(2,\C)_{\sigma \rho}$ denote the group of smooth maps $\gamma: \SSS^1 \to SL(2,\C)$, 
that are fixed by the involutions $\sigma$ and $\rho$ given by 
\[
(\sigma \gamma)(\lambda)= \Ad _P \gamma(-\lambda),
\quad (\rho \gamma)(\lambda) = (\overline{\gamma(\bar \lambda)}^t)^{-1}.
\]
where $P=\hbox{diag}(-1,1)$, and $\lambda$ is the $\SSS^1$ parameter. All loops considered here extend holomorphically to
$\C \setminus \{0\}$, and the reality condition given by $\rho$ means that
they take values in $SU(2)$ for 
real values of the loop parameter $\lambda$. We also consider the subgroups $\mathcal{G}^\pm$
consisting of loops the Fourier expansions of which are power series in $\lambda^{\pm 1}$.
We denote the corresponding Lie algebras by $\hbox{Lie}(\mathcal{G})$, $\hbox{Lie}(\mathcal{G}^\pm)$.
\begin{definition}
Let $I_x$ and $I_y$ be two real intervals, 
with coordinates $x$ and $y$, respectively. 
A potential pair $(\hat \chi, \hat \psi)$ 
is a pair of smooth $\hbox{Lie}(\mathcal{G})$-valued 
$1$-forms on $I_x$ and $I_y$ respectively with Fourier expansions in $\lambda$ 
as follows{\rm:}
\[
\hat \chi=\sum_{j=-\infty}^{1}\chi_{j}\lambda^{j}\>\mathrm{d}x,
\quad
\hat \psi=\sum_{j=-1}^{\infty}\psi_{j}\lambda^{j}\>\mathrm{d}y.
\]
\end{definition}
We will call the potential pair  {\em semi-regular} at a point $p$ if at least one of
the ``leading coefficients'' $\chi_1$ and $\psi_{-1}$ is non-zero at $p$, and \emph{regular}
if both are non-zero, and the potential pair is called (semi-)regular if the condition
holds at every point.

An admissible frame $\hat F$ is then obtained by solving
 $\hat X^{-1} \dd \hat X = \hat \chi$, and  $\hat Y^{-1} \dd \hat Y = \hat \psi$
for $\hat X(x)$ and $\hat Y(y)$, each  with 
initial condition the identity matrix, thereafter performing, at each $(x,y)$, a 
\textit{Birkhoff decomposition} (see \cite{PreS,jgp}):  
\beq \label{birkhoff1}
\hat X^{-1}(x) \hat Y(y) = \hat H_-(x,y) \hat H_+(x,y), 
\quad \textup{with} \quad \hat H_\pm(x,y) \in \mathcal{G}^\pm,
\eeq
and finally defining $\hat F$ by:
\beq  \label{Fdef}
\hat F(x,y) = \hat X(x) \hat H_-(x,y).
\eeq
  The admissible frame is
semi-regular if and only if the potential pair is semi-regular, and 
\emph{weakly} regular if and only if the potential pair is regular.

Conversely, any admissible frame $\hat F$ is associated to a potential pair
$(\hat X_+^{-1} \dd \hat X_+, \hat Y_-^{-1} \dd \hat Y_- )$, where $\hat X_+$ and $\hat Y_-$ are
obtained by the pair of pointwise \emph{normalized} Birkhoff factorizations
\beqas
\hat F = \hat X_+ \hat G_-, \quad \quad \hat X_+(x) \in \mathcal{G}^+, \,\,\hat G_-(x,y) \in \mathcal{G}^-,
      \quad \hat X_+ \big|_{\lambda=0}=I, \\
\hat F  = \hat Y _-\hat G_+,  \quad \quad \hat Y_-(y) \in \mathcal{G}^-, \,\,
     \hat G_+(x,y)\in \mathcal{G}^+, \quad \hat Y_-\big|_{\lambda=\infty}=I.
\eeqas
Note that the special form of an admissible connection automatically implies that $\hat X_+$ and $\hat Y_-$
depend only on $x$ and $y$ respectively.
Because of the normalization, these potentials are uniquely determined by $\hat F$ and
 have particularly simple forms:
\[
\hat X_+^{-1} \dd \hat X_+ = \bbar 0 & \zeta (x) \\ - \overline{\zeta (x)} & 0 \ebar \lambda \dd x, \quad
\hat Y_-^{-1} \dd \hat Y_- = \bbar 0 &  \xi (y) \\ - \overline{\xi (y)} & 0 \ebar \lambda^{-1} \dd y,
\]
and are called \emph{normalized potentials}.

\subsection{Potentials for non-characteristic singularities} \label{singpotentialsect}
Given the d'Alembert representation just described, a
generalized pseudospherical surface is locally determined by an arbitrary pair of 
(real)-differentiable complex-valued 
functions $\zeta (x)$ and $\xi(y)$. A generic function $\real \to \C$ is
 non-vanishing, and
so a generic normalized 
potential pair is regular, and the corresponding pseudospherical surface
is a wave front.

Our aim here is to give potentials that produce \emph{prescribed} singular curves.
We will consider separately two cases: that the singular set is or is not a characteristic
curve, starting with the non-characteristic case. 
For this, rather than normalized potentials, a better choice is a form of the 
\emph{boundary potential} pairs, introduced in \cite{dbms1} for the purpose of giving
prescribed values of $\hat F$ along a non-characteristic curve.  We assume that the singular
curve is non-degenerate and never parallel to a null curve.
Then we can always find local box coordinates $(x,y)$ such
that the curve is given by 
\[
y=\varepsilon x,  \quad \quad \varepsilon = \pm 1.
\]
Suppose given the value for $\hat F(x,y)$,  along the curve $y=\varepsilon x$.
In the coordinates
\[
u=\frac{1}{2}(x+\varepsilon y), \quad v= \frac{1}{2}(x-\varepsilon y),
\]
the curve is given by $v=0$,  and the value of $\hat F$ along the curve is given by 
\[
\hat F_0(u) = \hat F(u,0).
\]
Since $\hat F$ is assumed to be an admissible frame we have, from Definition \ref{admissibledef},
\beq  \label{FzeroMCform}
\hat F_0^{-1} \dd \hat F_0 = 
(\varepsilon V_\mathfrak{p} \lambda^{-1} + U_{\mathfrak{k}}+\varepsilon V_{\mathfrak{k}} + U_{\mathfrak{p}} \lambda ) \dd u.
\eeq
Since the highest and lowest powers of $\lambda$ appearing are $1$ and $-1$ respectively, this $1$-form 
is valid as either $\hat \chi$ or $\hat \psi$ or both  in a potential
pair. Hence, setting
\[
\hat X(x) = \hat F_0(x), \quad \hat Y(y)= \hat F_0(\varepsilon y),
\]
gives a valid potential pair $(\hat X^{-1} \dd \hat X, \hat Y^{-1} \dd \hat Y)$, called the 
\emph{boundary potential pair} relative to the curve $v=0$.
For this potential pair, the Birkhoff decomposition \eqref{birkhoff1} is trivial along
the curve $v=0$, since $\hat X(v=0) = \hat Y(v=0)$, and so the admissible frame $\widetilde F$
obtained by \eqref{Fdef} agrees with $\hat F$ along this curve. A uniqueness argument using
normalized potentials (see \cite{dbms1}) then shows that $\widetilde F$ and $\hat F$ determine the same harmonic map.

We now want to construct $\hat F_0(u)$ along a curve $v=0$ from geometric data of a 
pseudospherical frontal $f$ prescribed along the
curve.   Since the curve is non-characteristic, and assumed non-degenerate,
$f$ is necessarily a \emph{wave front} (see Section \ref{charsingsection}).
From Section \ref{extendedframesect}, we can assume that we are
 given box coordinates $(x,y)$ that are asymptotic coordinates  for $f$, the angle $\phi$
is the oriented angle between $f_x$ and $f_y$ and the first and second fundamental forms are:
\[
I = A^2 \dd x^2 + 2 \cos(\phi) A B \dd x \dd y + B^2 \dd y^2, \quad
II = 2A B  \sin(\phi) \dd x \dd y,
\]
where $A=|f_x|$ and $B=|f_y|$.
Using the same frame $F$ as in the proof of Lemma \ref{frontlemma}, defined by
\eqref{framedef}, we have:
\[
U_\mathfrak{k}=-\phi_x e_3, \quad U_\mathfrak{p}= A e_1, \quad V_\mathfrak{k}=0, \quad
V_\mathfrak{p} = -B(\cos \phi e_1 + \sin \phi e_2).
\]
In the coordinates $(u,v)$ we have $\phi_x=\frac{1}{2}(\phi_u+\phi_v)$.
If $v=0$ is a singular curve, we have $\phi=k \pi$ constant along the curve, so $\phi_u(u,0)=0$.
Without
loss of generality, we take $k=0$, i.e. $\phi(u,0)=0$.  The basic data that determine the 
boundary potential are thus 
\[
U_\mathfrak{k}=-\frac{\phi_v(u,0)}{2} e_3,  \quad  U_\mathfrak{p}=A(u) e_1, \quad 
V_\mathfrak{k}=0, \quad V_\mathfrak{p}=-B(u) e_1,
\]
 where $A(u)=|f_x(u,0)|$ and 
$B(u)=|f_y(u,0)|$. 
\begin{figure}[ht]
\centering
\begin{small}
$
\begin{array}{ccc}
\includegraphics[height=28mm]{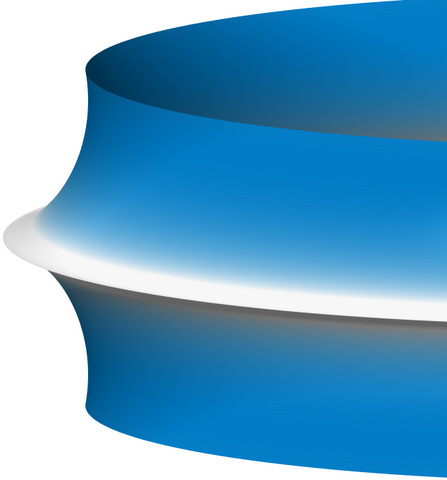}   \quad & \quad
\includegraphics[height=28mm]{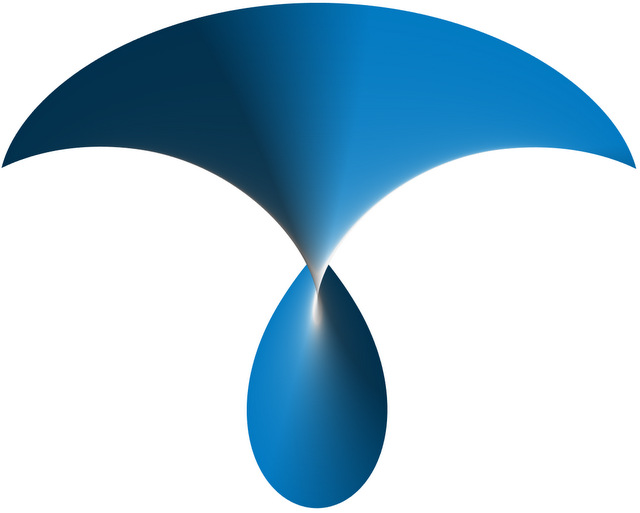}  \quad & \quad
\includegraphics[height=28mm]{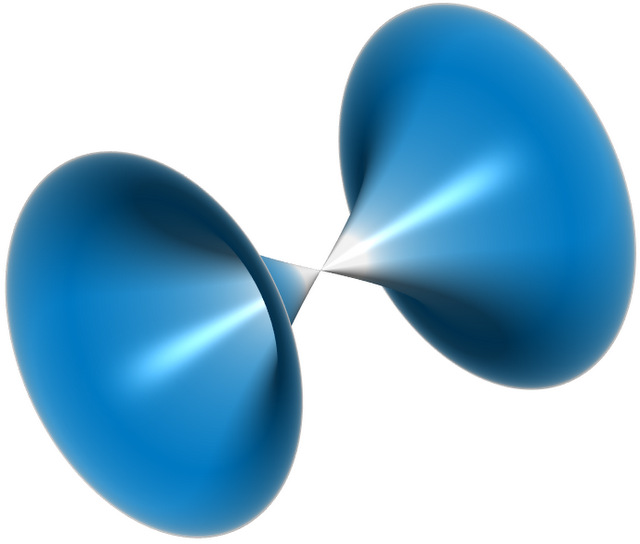}   \\
(A, \varepsilon B)=(1, 1)  &
(A,\varepsilon B)=(1+t, -1+t) &
(A,\varepsilon B) = (1, -1)
\end{array}
$
\end{small}
\caption{Non-degenerate singularities:  cuspidal edge, swallowtail and cone. } \label{singularitiesfig}
\end{figure}   
Substituting into \eqref{FzeroMCform}, and applying 
Theorem \ref{genericsingthm}, we conclude that all non-degenerate non-characteristic singular curves on 
pseudospherical frontals are obtained from the following theorem:
\begin{theorem}  \label{mainthm1}
Let $J$ be an open interval, $A, \,\, B: J \to (0,\infty)$ and $\beta: J \to \real$ three differentiable functions.
 Let $\varepsilon = \pm 1$ and set 
\[
\hat \eta := \left(-\varepsilon B(t) e_1 \lambda^{-1} -\frac{\beta(t)}{2} e_3 + A(t) e_1 \lambda \right) \dd t.
\]
Consider the potential pair $(\hat \eta, \hat \eta)$ defined on the intervals $I_x=J$ and $I_y=\varepsilon J$.  Let $f:I_x \times I_y \to \real^3$ be the generalized pseudospherical surface
obtained from $(\hat \eta, \hat \eta)$ via the generalized d'Alembert method. Then 
\begin{enumerate}
\item The set  $C:=\{y=\varepsilon x\}$ is a singular set for $f$.  
\item $C$ is non-degenerate at  a point $(x_0,\varepsilon x_0)$ if and only if $\beta(x_0) \neq 0$. In this case
    \begin{enumerate}
    \item  $C$ is diffeomorphic to a  cuspidal edge in a neighbourhood of $(x_0,\varepsilon x_0)$ if and only if $A(x_0)+\varepsilon B(x_0) \neq 0$.
    \item  $C$ is diffeomorphic to a swallowtail in a neighbourhood of $(x_0,\varepsilon x_0)$ if and only if  $A(x_0)+\varepsilon B(x_0) = 0$ and 
$A^\prime(x_0) + \varepsilon B^\prime(x_0) \neq 0$.
 \item $C$ is diffeomorphic to a cone singularity if and only if $A(x)+\varepsilon B(x) \equiv 0$.
\end{enumerate}
\end{enumerate}
\end{theorem}
Three non-degenerate examples are computed in Figure \ref{singularitiesfig},
 all with $\beta(t)=2$.
Some degenerate examples are shown below in Figure \ref{inflectionfig1}.
\subsection{Prescribed non-characteristic cuspidal edges}
Theorem \ref{mainthm1} gives the boundary potential pair for the generic 
non-characteristic singularities
of pseudospherical surfaces, as well as cones.   We now adapt this to produce 
pseudospherical surfaces with a given curve in $\real^3$ as a singular curve. 
We treat the case that the curve  is regular in $\real^3$, which means that the
singular curve, where non-degenerate, must be a cuspidal edge.  

The geometric Cauchy problem for \emph{regular} pseudospherical surfaces was 
studied in \cite{dbms1}.  For a non-characteristic curve,
there is a unique immersed solution containing a given curve $\gamma$ and with the surface normal $N$
prescribed along the curve, with a regularity condition $\langle \gamma^\prime(t), N^\prime(t) \rangle \neq 0$.
For the non-characteristic \emph{singular} geometric Cauchy problem we replace the
regularity condition with a \emph{singularity} condition, 
$\langle \gamma^\prime(t), N^\prime(t) \rangle = 0$:  \\
\noindent \textbf{Non-characteristic singular geometric Cauchy data along an open interval $J$:} 
\begin{enumerate}
\item A regular curve $\gamma: J \to \real^3$;
\item A unit vector field $Z: J \to \SSS^2 \subset \real^3$, satisfying
\[
\langle Z(t), \gamma^\prime(t) \rangle = 0,
\quad  \quad \langle Z^\prime(t), \gamma^\prime(t) \rangle =0.
\]
\item Weak regularity condition:
\[ |\gamma^\prime(t) | \neq |Z^\prime(t)|.
\]
\end{enumerate}
The above conditions are necessarily satisfied along a non-characteristic singular curve
on a pseudospherical frontal. We also find that
 the singular curve is non-degenerate at a point if and only if the curvature
$\kappa$ of the curve $\gamma$ is non-zero at that point.  Adding this assumption 
then simplifies the above description of the geometric Cauchy data.  Suppose that
$\gamma(s)$ is parameterised by arc-length. Let ${\bf t}$, ${\bf n}$ and
${\bf b}$ be the Frenet-Serret frame along the curve. The vector field $Z$ must satisfy: $\langle Z, {\bf t} \rangle =0$
and $\langle Z^\prime, {\bf t} \rangle=0$. Differentiating the first equation gives
\[
\langle Z^\prime, {\bf t} \rangle = -\langle Z, {\bf t}^\prime  \rangle
  = -\kappa \langle Z, {\bf n}\rangle.
\]
Hence, the assumptions $\langle Z, {\bf t} \rangle =0$ and $\kappa \neq 0$ imply that
$\langle Z, {\bf n}\rangle=0$.  It follows that $Z=\pm {\bf b}$, where ${\bf b}$ is the
unit binormal to the curve. Since ${\bf b}^\prime = -\tau {\bf n}$, where $\tau$
is the torsion, the weak regularity condition $|\gamma^\prime| \neq |N^\prime|$
becomes $\tau \neq \pm 1$.  To simplify matters, we will also take $\tau>-1$.
Hence, for non-degenerate singular curves,  the geometric
Cauchy data is the curve given in the following result:
\begin{theorem} \label{mainthm2}
Let $\gamma: J \to \real^3$ be a regular arc-length parameterised curve,
with curvature $\kappa$ and torsion $\tau$ satisfying
\[ \kappa(s) \neq 0, \quad \hbox{and either} \quad  |\tau(s)|< 1, \quad \hbox{or}  \quad \tau(s) > 1 \]
along $J$.  Let $\varepsilon:= \hbox{sign}(\tau-1)$.  Then:
\begin{enumerate}
\item \label{item1}
There exists, unique up to a Euclidean motion, a pseudospherical wave front $f(u,v)$, with box coordinates
 $(x,y)$  and $u=(x+\varepsilon y)/2$, $v=(x-\varepsilon y)/2$,
containing
$\gamma$ as a non-characteristic singular curve in the form $f(u,0)=\gamma(u)$.
 The singular curve is non-degenerate.
\item \label{item2}
The surface $f$
is given by the d'Alembert method, with potential pair $(\hat \eta,  \hat \eta)$ on $J \times \varepsilon J$,
with 
\[
\hat \eta = \left(\frac{\tau-1}{2}e_1 \lambda^{-1} + \kappa e_3 + \frac{\tau + 1}{2} e_1 \lambda \right) \dd s.
\] 
\item  \label{item3}
All non-degenerate non-characteristic singular curves of pseudospherical frontals that have a regular image in $\real^3$
are obtained this way.
\end{enumerate}
\end{theorem}
\begin{proof}
By Theorem \ref{mainthm1} there is a generalized pseudospherical surface generated by any triple
of functions $A$, $B$ and $\beta$.  The surface is a wave front if and only if both $A$ and $B$
are non-vanishing, which, in this case means $\tau\neq \pm 1$. The non-degeneracy 
condition is $\beta = -2\kappa(t) \neq 0$.

Now suppose the existence of a  pseudospherical wave front  $f:J \times \varepsilon J \to \real$ with  $f(u,0)=\gamma(u)$ a non-degenerate non-characteristic singular curve.  As described above, it follows that the surface normal satisfies 
$N(u,0)= \pm {\bf b}(u)$. Since we are only looking for the potential up to a Euclidean motion, we can take
\[
N(u,0)= {\bf b}(u).
\]
 Along the singular curve, the vectors $f_u$, $f_v$, $f_x$ and $f_y$ are all parallel.
As previously, let $F$ be the frame defined at \eqref{framedef}, so that, on $v=0$,
\[
\Ad _F e_1 = \frac{f_x}{|f_x|}  = \frac{f_y}{|f_y|}, \quad \Ad _F e_3 = N ,
\]
which is to say that
\[
\Ad_F e_1 = f_u = \gamma^\prime,  \quad \Ad_F e_2 = {\bf n}, \quad \Ad _F e_3 = {\bf b}.
\]

We have already shown in Section \ref{singpotentialsect} that along $v=0$ 
\[
F^{-1} F_u = (-\varepsilon B(u) + A(u)) e_1 -\frac{\beta(u)}{2} e_3,
\]
where $A(u)=|f_x(u,0)|$ and $B(u)=|f_y(u,0)|$. 
Differentiating ${\bf b} = \Ad_F (e_3)$, we have
\beqas
{\bf b}^\prime &=& \Ad_F[F^{-1} F_u, e_3] \\
&=&   (\varepsilon B -A)  \Ad _F ( e_2 ),
\eeqas
so that
\[
\varepsilon B(u)-A(u) = - \tau(u).
\]
We also have $\gamma^\prime(u)=f_u = f_x + \varepsilon f_y$, from which
\[
1= A(u)^2+2\varepsilon A(u) B(u) + B(u)^2.
\]
There are, in general, two solutions for positive $A$ and $B$, but the surfaces obtained from the 
corresponding potentials are congruent after interchanging $x$ and $y$.
Hence we can take the solution:
\[
A = \frac{\tau +1}{2}, \quad \varepsilon B = \frac{1-\tau}{2}, \quad \varepsilon = \hbox{sign}(1-\tau).
\]
To find  $\beta$, we use 
\beqas
\kappa  {\bf n} = \gamma^{\prime \prime} &=&
   \Ad _F[F^{-1} F_u, e_1] \\
	&=& -\frac{\beta}{2}  \Ad_F ( e_2 ),
\eeqas
so $\beta=-2 \kappa$. 
Substituting the expressions for $A$, $B$, $\varepsilon$ and $\beta$ into the potential $\hat \eta$
of Theorem \ref{mainthm1} gives the potential in the theorem statement.
 Since the above data were obtained from an arbitrary solution
of the geometric Cauchy problem, this also proves uniqueness, and so items \eqref{item1} and \eqref{item2}
are proved.  Item \eqref{item3} follows from the fact, already explained, that, for a non-degenerate
non-characteristic singular curve the curvature is non-vanishing and the torsion satisfies $|\tau|\neq 1$.
\end{proof}


\begin{figure}[ht]
\centering
$
\begin{array}{ccc}
\includegraphics[height=30mm]{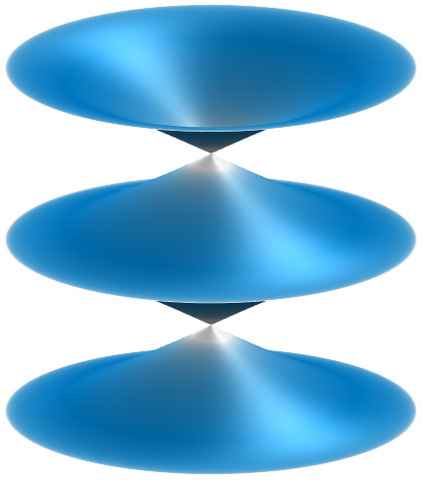}  \quad & \quad \quad
\includegraphics[height=30mm]{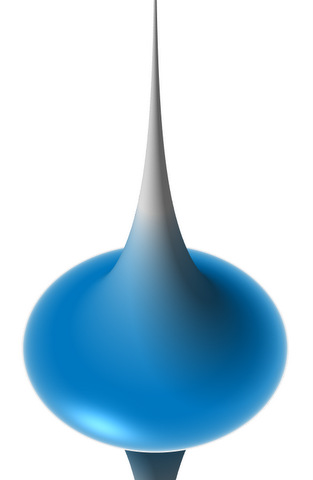}  \quad \quad & \quad
\includegraphics[height=30mm]{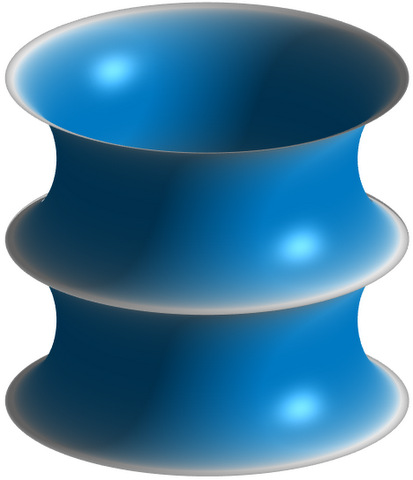}  
\end{array}
$
\caption{Example \ref{example1}, $R=0.5$, $R=1$ and $R=1.5$ . } \label{sorfig}
\end{figure}

\begin{example} \label{example1}
Circles: Take $\gamma(t)=R(\cos t, \sin t, 0)$, where $R>0$. The arc-length parameter, 
curvature and torsion are $s=Rt$, $\kappa=1/R$ and $\tau=0$. The potential is thus: 
\[
\hat \eta = \left(-\frac{R}{2}e_1 \lambda^{-1} + e_3 + \frac{R}{2}e_1 \lambda\right) \dd t,
\]
and this gives the well-known pseudospherical surfaces of revolution. The case $R=1$ is the pseudosphere.
\end{example}

\begin{example} \label{example2}
Helices:
Taking $\kappa$ and $\tau$ both constant, with $\tau\neq 0$, gives a surface containing a circular helix as a cuspidal edge
(Figure \ref{helixtypesfig}). 
 Helical,  as well as rotational, constant curvature surfaces, 
were studied by Minding in \cite{minding}.
These surfaces are generally periodic in the $v$ direction, which can 
be seen  by considering that  the curve is
invariant under a $1$-parameter family of rigid motions (a screw-motion). The surface
must also have this symmetry by uniqueness of the solution to the geometric Cauchy problem.
 Hence the next singular curve encountered when moving in the $v$ direction
is also a circular helix. By the symmetry of the initial data, it follows that every second singular
curve is congruent.  

\begin{figure}[ht]
\centering
$
\begin{array}{cccc}
\includegraphics[height=30mm]{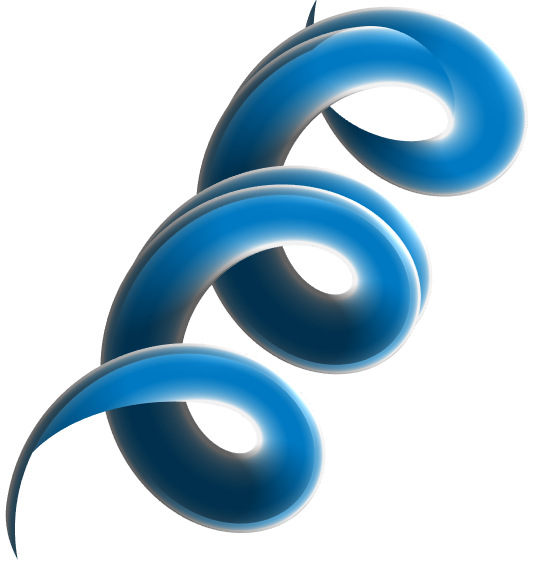}  & 
\includegraphics[height=30mm]{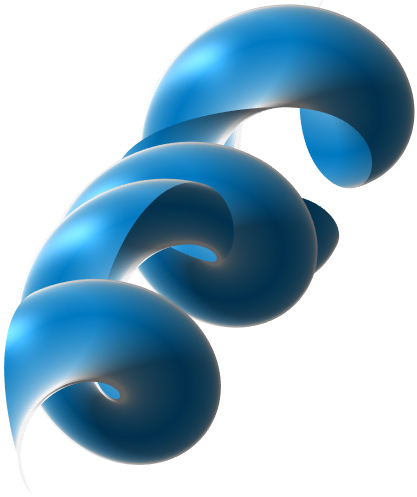}  &
\includegraphics[height=30mm]{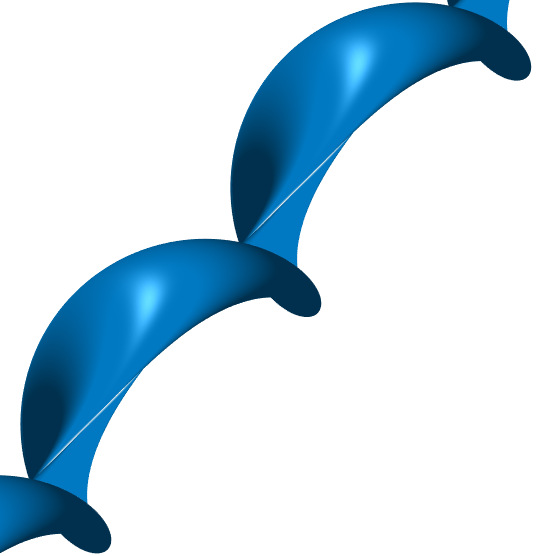}   & 
\includegraphics[height=30mm]{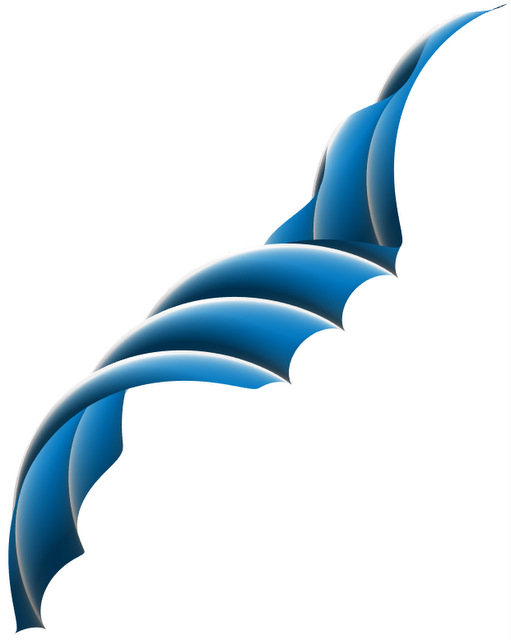}   \\
\kappa^2+\tau^2>1,  &
\kappa^2+\tau^2>1, &
\kappa^2+\tau^2=1 &  \kappa^2+\tau^2<1 \\
|\tau|>1  &  |\tau|<1  & & 
\end{array}
$
\caption{Examples of helical pseudospherical surfaces. } \label{helixtypesfig}
\end{figure} 
As with the case of the circle, there are essentially three types:
\begin{enumerate}
\item Case $\kappa^2+\tau^2>1$:  Here there are two sets of helices with the same axis but different
radius.  The initial curve is on the outer cylinder when $|\tau|>1$, and the inner when $|\tau|<1$.
\item  Case $\kappa^2+\tau^2=1$:   The special case where the inner helices degenerate to a 
straight line.   These are \emph{Dini's surfaces}, which can be parametrized as 
\[
f(\zeta,\xi)=(a \cos \zeta \sin \xi, a \sin \zeta \sin \xi, a(\cos \xi + \ln(\tan(\xi/2))) + b\xi),
\]
where, for the case of constant curvature $K=-1$, we must have $a^2+b^2=1$.
The surface has singularities at $\cos(\xi)=0$, so we can take the helix 
\[ 
\gamma(t) = f(t,\pi/2)=(a \cos t, a \sin t, bt)
\]
 as the initial curve.
We then have $\kappa=|a|$ and $\tau=b$.  Hence  Dini's surfaces are
given by constant $\kappa$ and $\tau$, with $\kappa^2 + \tau^2=1$.
\item Case $\kappa^2+\tau^2<1$:  Here the inner helix disappears completely, so that 
all singular curves are congruent.
\end{enumerate}
\end{example}

\begin{example} \label{cylindercurve}
The closed curve $\gamma(t)=(\cos(3t), \sin(3t), -\sin(t))$ lies on a round cylinder and has two self-intersections.
Computing $\kappa(t)=3(8\cos^2(t)+82)^{1/2}(\cos^2(t)+9)^{3/2}$, $\tau=-12\cos(t)/(4\cos^2(t)+41)$ and
$\dd s = \sqrt{\cos^2(t)+9} \dd t$, we see that $\kappa$ is non-vanishing and $|\tau|<1$. The 
surface that contains this curve as a cuspidal edge is shown in Figure \ref{cylfig}.
\end{example}

\begin{figure}[ht]
\centering
$
\begin{array}{cc}
\includegraphics[height=32mm]{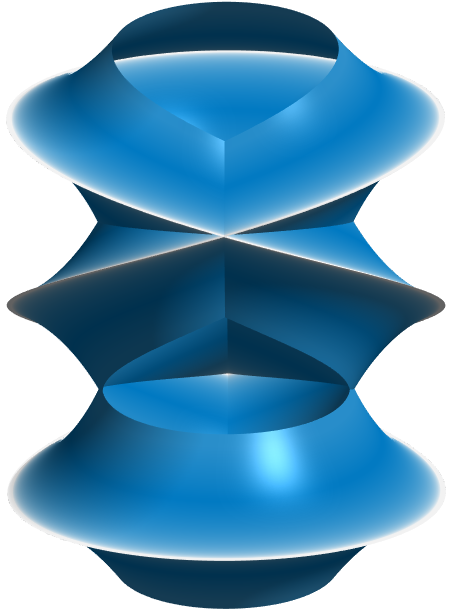}  \quad & \quad
\includegraphics[height=32mm]{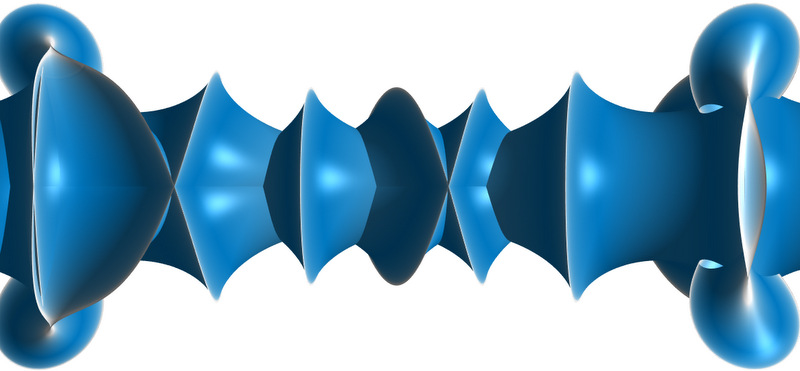} 
\end{array}
$
\caption{Example \ref{cylindercurve}.}  \label{cylfig}
\end{figure}

\begin{example} \label{degenexamples}
 Examples with inflections and with $\tau$ taking the value $1$:
Theorem \ref{mainthm2} is stated for curves with $\kappa$ non-vanishing and $\tau \neq \pm 1$.
However, we can use any functions $\kappa$ and $\tau$ and still obtain a valid potential pair, and 
therefore a pseudospherical frontal.  If we take $\kappa \equiv 0$, the solution degenerates to
a straight line.  If we take $\tau \equiv \pm 1$, the solution degenerates to a helix curve.

\begin{figure}[ht]
\centering
$
\begin{array}{cc}
\includegraphics[height=28mm]{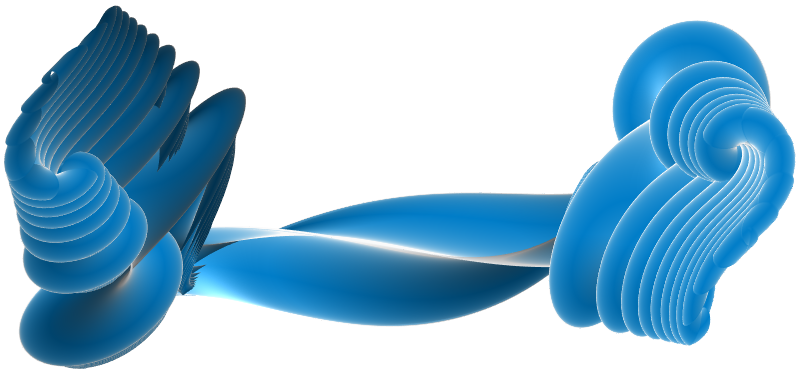}    \quad &  \quad
		    	\includegraphics[height=32mm]{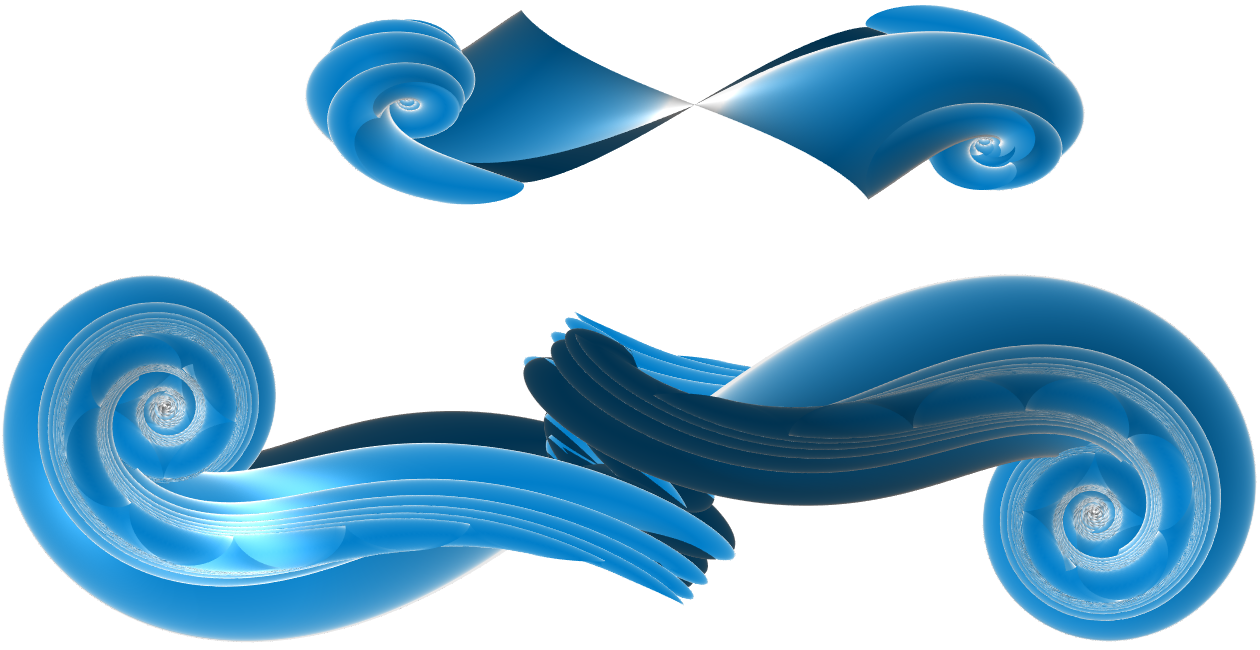}    \\
	\kappa(t)=t, \,\,\, \tau(t)=1/2.  & \kappa(t)=t, \,\,\, \tau(t)=0.
\end{array}
$
\caption{Singular curves with inflections.} \label{inflectionfig1}
\end{figure}

 If $\kappa$ vanishes at just one point we will get a singular curve that is degenerate
at this point, but non-degenerate elsewhere, provided $|\tau|\neq 1$.  The most basic example is 
$\kappa(t)=t$, $\tau(t)=1/2$,  shown in Figure \ref{inflectionfig1}.
At the point $(0,0)$, there are two cuspidal edges crossing each other. 
For the example $\kappa(t)=t$, $\tau(t)=0$ the surface appears to have a degenerate cone point.
The case $\kappa(t)=t^2$, $\tau(t)=1/2$ is also computed and shown in Figure \ref{figureunbounded}. In this case, the
singular set is a single curve through the point $(0,0)$.

If we take $\tau =\pm 1$ at just one point, the surface is not a wave front at this point.
This is because the potential pair is $(\hat \chi, \hat \psi)$, where
\beqas
\hat \chi  &=&
      \frac{1}{2}\left((\tau(x)-1)e_1 \lambda^{-1} + 2 \kappa(x) e_3 + (\tau(x)+ 1)e_1 \lambda \right) \dd x,\\
\hat \psi  &=&
        \frac{1}{2} \left((\tau(y)-1)e_1 \lambda^{-1} + 2 \kappa(y) e_3 + (\tau(y) + 1) e_1 \lambda \right) \dd y,
\eeqas
so exactly one of $\chi_1$ and $\psi_{-1}$ vanishes. The potential pair is semi-regular but not regular.
Moreover, the singular curve must be degenerate at this point, because we showed 
in Section \ref{charsingsection} that if the singular curve is non-degenerate at a point
where the surface is not a wave front, then the curve is a characteristic curve on a \emph{neighbourhood}
of this point, which is not the case here.
The surface shown at Figure \ref{viviani_fig} is generated by the Viviani figure-8 space curve
$\gamma(t)=0.3(1+ \cos(t), \sin(t), 2 \sin(t/2)$.  The torsion takes the values $\pm 1$ twice each,
and at each such point another singular curve branches off from the figure eight 
(Figure \ref{viviani_fig}, right).

\end{example}

\section{Prescribed characteristic singular curves}  \label{charsection}
Now we want to give potentials for non-degenerate \emph{characteristic}
singular curves.   As expected for a Cauchy problem along a characteristic, we will find that 
data along a curve does not specify a unique solution: 
further data must be provided along another, transverse, characteristic curve.
Moreover, with our solution, the non-degeneracy is only guaranteed
in a neighbourhood of the intersection of these two curves.

As explained in Section \ref{charsingsection},  given that the map is semi-regular,
we can assume that box coordinates are chosen such 
that the singular curve is locally given as $\{y=0\}$, and  can choose a local frame satisfying
\bdm
f_x =  \Ad_F e_1, \quad f_y = \Ad_F(ae_1 + be_2), \quad
 f_x \times f_y =  b N, \quad N= \Ad_Fe_3,
\edm
where 
\[
b(x,0)=0, \quad \frac{\partial b}{\partial y}(x,0) \neq 0.
\]
The surface is a wave front at points where $a(x,0) \neq 0$.
  The curve $\gamma(x)=f(x,0)$
is already arc-length parameterised. Hence, differentiating the expression for $f_x$ we
have:
\[
 f_{xx}=\Ad_F[u_0 e_3 + e_1, e_1] = u_0\Ad_F e_2,
\]
along $y=0$.  Thus, up to a change of orientation, $u_0(x,0)= \kappa(x)$, 
the curvature of $\gamma$. Note that if 
$\kappa(x) \neq 0$ for all $x$ then the curve has a well defined 
normal ${\bf n}=\Ad_F e_2$ and hence the binormal is ${\bf b}= \Ad_F e_3=N$. We then have
\[
-\tau n = \frac{\dd {\bf b}}{\dd x} = \Ad_F[\kappa e_3 + e_1, e_3] = -\Ad_F e_2,
\]
from which we conclude that $\tau(x)=1$ along the whole curve. Although the curve
is singular, this is the same property that asymptotic curves (of non-vanishing curvature)
have on a regular pseudospherical surface, namely that $\tau=\pm 1$.
 
Now differentiating the expression $f_x \times f_y = b \Ad_Fe_3$, using $b(x,0)=b_x(x,0)=0$,
we also have
\[
0=-\kappa(x)a(x,0)\Ad_Fe_3.
\]
Hence, if the surface is a wave front we must have $\kappa(x)=0$ for all $x$. In other words, the only possible
non-degenerate characteristic singular curve on a pseudospherical \emph{wave front} is a straight
line.

\begin{theorem}  \label{mainthm3}
Let $I_x$ be an open interval containing $0$, and $\gamma: I_x \to \real^3$  a regular space curve, parameterised by arc-length,
with either non-vanishing curvature function $\kappa$, and constant torsion $\tau=\pm 1$,
or with curvature everywhere zero on $I_x$.
 Let $I_y$ be an open interval containing $0$.
For every choice of differentiable $1$-form of type
\[
 \hat \psi = (\alpha(y) e_1 + \beta(y) e_2) \lambda^{-1} \dd y.
\]
 with
\[
\beta(0)=0, \quad \quad \beta^\prime(0) \neq 0,
\]
and 
\[
\alpha(0)=0, \quad \hbox{if} \quad \kappa \not\equiv 0,
\]
 there corresponds a unique  pseudospherical frontal $f:I_x \times I_y \to \real^3$,  such that
\begin{enumerate}
\item  \label{itema}
$f$ is semi-regular on an open set containing $I_x \times\{0\}$, and
\item \label{itemb}
$f(x,0)=\gamma(x)$ is a characteristic singular curve in the surface, non-degenerate on a neighbourhood
of $(0,0)$. 
\end{enumerate}
  Up to a Euclidean motion,
the surface $f$ is given by the 
d'Alembert method with potential pair $(\hat \chi, \hat \psi)$ on $I_x \times I_y$,
where
\[
\hat \chi = (\kappa(x) e_3 + \lambda e_1) \dd x,
	\]
	and all such surfaces $f$ satisfying \eqref{itema} and \eqref{itemb} are obtained this way.
\end{theorem}
\begin{proof}
The $1$-forms defined satisfy the requirements for a potential pair, and therefore
integrating $\hat X^{-1} \dd \hat X = \hat \chi$, and $\hat Y^{-1} \dd \hat Y = \hat \psi$,
with initial conditions $\hat X(0)=I$  and $\hat Y(0)=I$, performing
a Birkhoff decomposition
\[
\hat X^{-1}(x) \hat Y(y) = \hat H_-(x,y) \hat H_+(x,y), \quad
\hat H_\pm(x,y) \in \mathcal{G}^\pm, \quad \hat H_-(x,y)\big|_{\lambda=\infty}=I,
\]
gives us an admissible frame $\hat F= \hat X \hat H_- = \hat Y \hat H_+^{-1}$. 
We write $O_\pm(\lambda^{\pm k})$ for any convergent Fourier series of the form
$\sum_{j=k}^\infty  a_k \lambda^{\pm j}$.
 The normalization of $\hat H_-$
means that its Fourier expansion is $\hat H_- = I + O_-(\lambda^{-1}) $, so
\[
\hat F^{-1} \dd \hat F =  \lambda e_1 \dd x + O_-(1).
\]
Since the coefficient of $\lambda$ is $e_1 \dd x$, 
we can apply the analysis of Section \ref{charsingsection} to conclude that
\[
\hat F^{-1} \dd \hat F  = (u_0 e_3 + \lambda e_1) \dd x + 
    (a(x,y)e_1 + b(x,y)e_2)\lambda^{-1} \dd y.
\]
 Along the curve $y=0$ we have
$\hat Y = I$, and so the unique factor $\hat H_-$ in the Birkhoff decomposition above satisfies
$\hat H_-(x,0)=I$.  Thus $\hat F(x,0)=\hat X(x)$, and, along $y=0$ we have
\[
\hat F^{-1}  \hat F_x 
    = (\kappa(x) e_3 + \lambda e_1).
\]
Hence
\[
u_0(x,0)=\kappa(x).
\]
To check the non-degeneracy condition on $\partial_y b(x,0)$,  we  will use the expression
$\hat F= \hat Y  \hat H_+^{-1}$.  Since $\hat H_+$ is $\mathcal{G}^+$-valued, we can write
\[
\hat H_+ = D_0 + O_+(\lambda),  \quad
  D_0= \hbox{diag}\left(e^{i \theta/2} ,  e^{-i \theta/2} \right).
\]
We have $\hat H_+^{-1}(x,0)=\hat X(x)$, and so, along $y=0$,
\[
\frac{\theta_x(x,0)}{2} e_3 +O_+(\lambda) = \hat H_+ \frac{\partial H_+^{-1}}{\partial x} = 
\hat X^{-1} \frac{\partial \hat X}{\partial x} = \kappa(x) e_3 + \lambda e_1,
\]
whilst along $x=0$, we also have $\hat H_+(0,y)=I$. Hence 
\[
\theta_x(x,0) = 2\kappa(x), \quad \quad \theta_y(0,y) = 0.
\]
 From $\hat F= \hat Y  \hat H_+^{-1}$ we obtain
\[
\hat F^{-1} \dd \hat F = \Ad_{D_0}(\alpha e_1 + \beta e_2) \lambda^{-1} + O_+(\lambda),
\]
which gives
\[
b(x,y) = \cos(\theta(x,y)) \beta(y) + \sin(\theta(x,y)) \alpha(y).
\]
Differentiating this, using $\beta(0)=0$:
\[
\frac{\partial b}{\partial y}(x,0) =  \alpha(0) \theta_y \cos \theta 
 +\beta^\prime (0)  \cos \theta   +   \alpha^\prime(0) \sin \theta.
\]
For the case $\kappa(x)\equiv 0$, we have $\theta_x(x,0)=0$, so $\theta$ is constant along $x=0$, and 
$\cos(\theta(x,0)) = 1$, $\sin(\theta(x,0)=0$ by the initial condition at $(0,0)$. Thus,
\[
\frac{\partial b}{\partial y}(x,0)   =  \alpha(0) \theta_y(x,0)  +\beta^\prime (0).
\]
Since $\theta_y(0,0)=0$ and $\beta^\prime(0) \neq 0$, it follows that the non-degeneracy condition
$b_y(x,0) \neq 0 $ is satisfied on an open set containing $(0,0)$. 
On the other hand, for the case $\kappa \neq 0$, where we take $\alpha(0)=0$,  we have
\[
\frac{\partial b}{\partial y}(x,0)  =\beta^\prime (0)  \cos (\theta(x,0))   +   \alpha^\prime(0) \sin (\theta(x,0)).
\]
In this case,  we use $\cos(\theta(0,0)) = 1$, $\sin(\theta(0,0)=0$ to again conclude that
$b_y(x,0) \neq 0 $ is satisfied on an open set containing $(0,0)$. 

To see that  the singular curve 
$f(x,0)$, of the solution $f$, coincides with $\gamma$,
the discussion preceding the statement of this theorem shows $f(x,0)$
 has curvature $\kappa$ and, if $\kappa$ is non-vanishing,
 constant torsion $\tau=1$.  Since a curve is
determined by its curvature and torsion, we must have,  up to 
a Euclidean motion, $f(x,0)=\gamma(x)$.  If $\kappa$ is everywhere zero, then
the curve is just a straight line segment of the same length as $I_x$,
again identical with $\gamma(x)$ up to a Euclidean motion.

For uniqueness given the potential $\hat \psi$, it is enough to observe
that $\hat \psi$ is a normalized potential, with normalization point $(0,0)$, which
is uniquely determined by the surface $f: I_x \times I_y \to \real^3$ and the choice of
normalization point.
  Thus, given any surface $\tilde f$ satisfying $\tilde f(x,0)=\gamma(x)$, we 
obtain $\hat \chi$ from the knowledge of $\kappa$, and the frame $\tilde{\hat F}(x,0)$, 
and we recover $\hat \psi$ from a normalized Birkhoff decomposition of
 $\tilde{\hat F}(x,y)$ as described at the end of Section \ref{gdamethod}.
Hence $\tilde f = f$.  Since $\hat \psi$ is the most general normalized potential satisfying
the regularity conditions,   all possible solutions are obtained this way.
\end{proof} 

\begin{remark}
\begin{enumerate}
\item Because  $\beta(0)=0$ and $\beta^\prime(0) \neq 0$, we can, on a neighbourhood of $y=0$, change
$y$-coordinates to $\tilde y(y)$ so that $\beta(y) \dd y = \tilde y \dd \tilde y$.  In these coordinates the potential
$\hat \psi$ is of the form
\[
\hat \psi = (\tilde \alpha(\tilde y) e_1 + \tilde y e_2) \lambda^{-1} \dd \tilde y.
\]
Thus, given $\kappa$, the unique solution is determined, on an open set containing the curve,
 by a single function $\tilde \alpha(\tilde y)$ that is arbitrary if $\kappa \equiv 0$ but, in the general 
case must  satisfy $\tilde \alpha(0)=0$ 
\item  
For the case that $\kappa(x) \equiv 0$,  adding the assumption 
$\alpha(0)=0$ guarantees that the entire singular curve is non-degenerate.
\item
Suppose  coordinates are  chosen such that $\beta(y)=y$, as just described. Then, if $\alpha$ is an odd function of $y$ the surface
has a \emph{fold} singularity along $y=0$, i.e $f$ satisfies $f(x,y)=f(x,-y)$.  This can be seen from the symmetry
$\hat \psi(-y)= \hat \psi(y)$.
  Such a singularity, at least if $\alpha$ is analytic, can be ``removed" in the sense that one half
of the folded surface is part of a regular pseudospherical surface which contains the same curve:
 writing $\alpha(y)=y(a_1 + a_3 y^2 + \dots )$, and
 setting $2 \tilde y= y^2$, we have, for $y>0$,  the expressions
 $\tilde \alpha(\tilde y) \dd \tilde y= \alpha(y) \dd y = (a_1 +  a_3 2 \tilde y + a_5 (2\tilde y)^2 + \dots) \dd \tilde y$
and $y \dd y = \dd \tilde y$. Hence the surface corresponding to the pair
 $\hat \psi = (\tilde \alpha(\tilde y) e_1 + e_2) \lambda^{-1} \dd \tilde  y$ and
$\hat \chi = (\kappa(x) e_3 + \lambda e_1) \dd x$ is regular on an open set containing the $x$-axis and
agrees with the folded surface on the set $y>0$.  

 Of course the Lorentz structure corresponding to the
two surfaces are different here at the line $y=0$.  For a given global Lorentz structure there is no way to 
remove this singularity because the vanishing of a $1$-form $g(y) \dd y$ is well defined with respect to changes
of box-charts.  An example of a folded Amsler surface is shown in
Figure \ref{linesfig2}.
\end{enumerate}
\end{remark}

\begin{example}  
Weakly regular characteristic singularities:  These are all given by data of the form $\kappa \equiv 0$, $\beta(y)=y$ and
an arbitrary choice of $\alpha$ with $\alpha(0) \neq 0$.   The singular curve is guaranteed to be non-degenerate in 
a neighbourhood of $(0,0)$.  An example is shown in Figure \ref{linefig}. 
\end{example}

\begin{figure}[ht]
\centering
\begin{small}
$
\begin{array}{ccc}
\includegraphics[height=26mm]{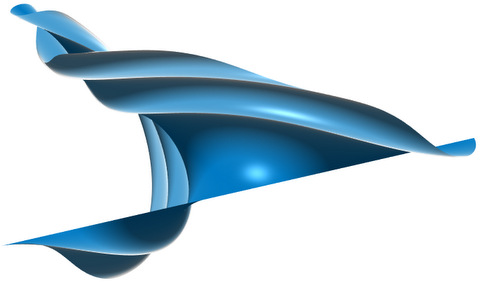}   & \quad
\includegraphics[height=26mm]{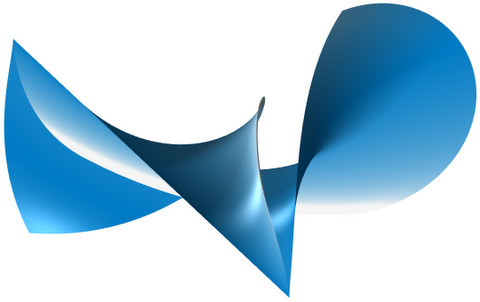}  
 \quad &
\includegraphics[height=26mm]{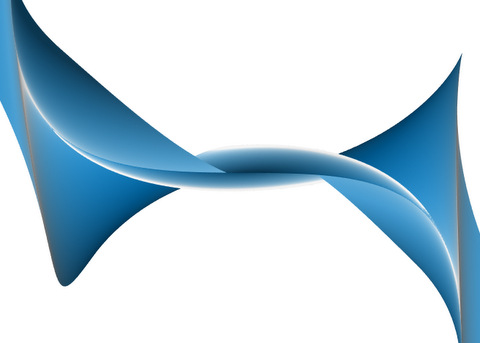} \\
\kappa=0, \,\, \alpha(y)=0 &
\kappa=0, \,\,  \alpha(y)=y^2 & \kappa=1, \,\, \alpha(y)=y^2
\end{array}
$
\end{small}
\caption{Non-weakly regular singular curves. 
 Left: folded Amsler surface. Middle: higher order
cuspidal edge.  Right: Spiral singularity. All have $\beta(y)=y$.
( Examples \ref{nonfrontlinesex}  and \ref{deg_spiralex}).} \label{linesfig2}
\end{figure}   

\begin{example}  \label{nonfrontlinesex} 
Straight lines that are not weakly regular: These are given by $\kappa \equiv 0$, $\beta(y)=y$ and any choice of $\alpha$
with $\alpha(0)=0$.   The entire line is a non-degenerate singularity.  These are all higher order cuspidal edges.   See Figure \ref{linesfig2}.  
If $\alpha$ is an odd function, we have a fold.  If $\alpha$ is not an odd function then we cannot ``remove" the singular
curve as can be done with the fold.  For example, for the case $\alpha(y)=y^2$ and $\beta(y)=y$, 
let $S_+$ denote the surface generated by $(\hat \chi, \hat \psi)$, for $y > 0$.  Then $S_+$ does not extend to
a pseudospherical wave front over the curve $y=0$. If it did, because asymptotic directions are
well defined on a pseudospherical surface, the surface would be generated by
a potential pair  $(\hat \chi, \tilde {\hat \psi})$, where $\hat \chi$ is unchanged and the one-form $\tilde {\hat \psi}$ agrees with
$\hat \psi$ on the set $y \geq 0$, but where $\tilde {\hat \psi}$ is regular at $y=0$.  In other words, we are looking for 
a change of coordinates $\tilde y(y)$ valid on $y>0$ such that the $1$-form $(y^2 , y) \dd y= (y^2 \frac{\dd y}{\dd \tilde y} , y\frac{\dd y}{\dd \tilde y} ) \dd \tilde y$ extends to a regular $1$-form  at $y=0$.  By definition, this means that 
both components are smooth and at least one
 non-vanishing at $y=0$.  If $y \frac{\dd y}{\dd \tilde y}$ is non vanishing, we can assume that $\tilde y$ is chosen so that 
$y \frac{\dd y}{\dd \tilde y} = 1$, that is $\tilde y = y^2/2$, and hence $(y^2, y) \dd y =  (\sqrt{2 \tilde y}, 1)  \dd \tilde y$,
which is not differentiable at $\tilde y=0$.  A similar argument shows that coordinates cannot be found such that
the first component $y^2 \dd y$ is non-zero.
\end{example}

\begin{example}  \label{deg_spiralex}
Figure \ref{linesfig2} (right) shows a  pseudospherical frontal that contains a helix curve. The surface is not
a wave front because the singular curve is characteristic and not a straight line.
 The singularity is non-degenerate in a neighbourhood
of $(0,0)$, but degenerates at some points, which can be seen where it is intersected by other singular curves.
\end{example}
\begin{example}  \label{frontlinesex}
Weakly regular characteristic singularities:  These are all given by data of the form $\kappa \equiv 0$, $\beta(y)=y$ and
an arbitrary choice of $\alpha$ with $\alpha(0) \neq 0$.   The singular curve is guaranteed to be non-degenerate in 
a neighbourhood of $(0,0)$.  An example is shown in Figure \ref{linefig}. 
\end{example}

\section{Examples and numerics}  \label{numericssection}
In this section we use numerics to give a picture of some degenerate singularities, as well as 
to show the global appearance of solutions generated by certain types of singular curve.
\subsection{Degenerate singularities}
In Example \ref{degenexamples} we saw some  degenerate singularities where $\kappa$ vanishes
or $|\tau|$ takes the value $1$ on a singular curve generated by Theorem  \ref{mainthm2}.
Theorem  \ref{mainthm1} is slightly more general, and the condition for a degenerate singularity
for the potential $\hat \eta = (-B(t)e_1 \lambda^{-1} - \beta(t)/2 e_3 + A(t) e_1 \lambda) \dd t$
is that $\beta$ vanishes.  Two examples are shown in Figure  \ref{degen_example}.  
Both are degenerate cone points.

\begin{figure}[ht]
\centering
\begin{small}
$
\begin{array}{cc}
\includegraphics[height=28mm]{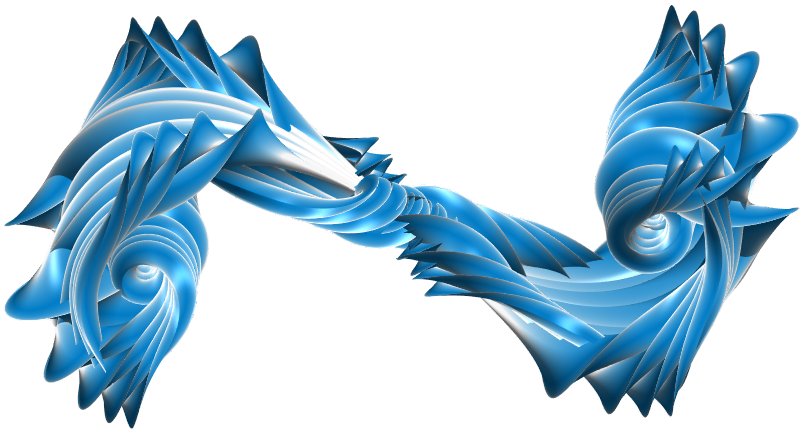}  
 \quad & \quad
\includegraphics[height=26mm]{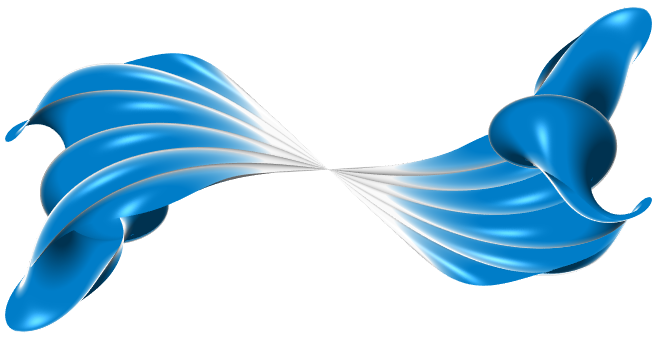}  
\end{array} 
$
\vspace{3ex}\\
$
\begin{array}{cc}
\includegraphics[height=28mm]{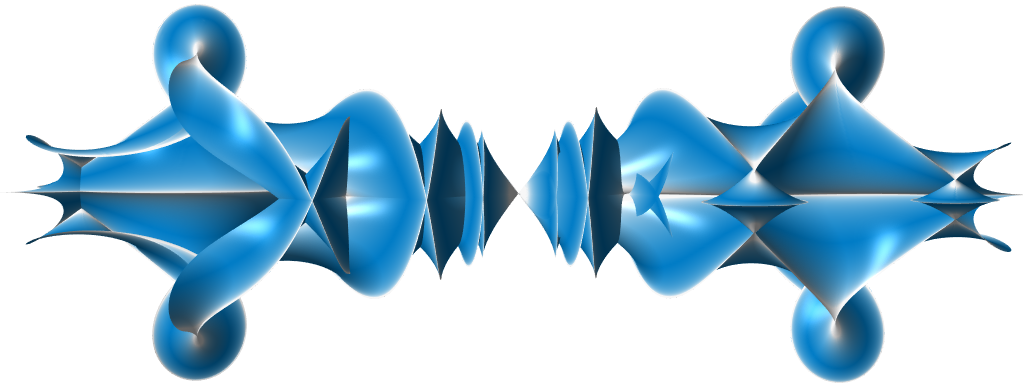}  
 \quad & \quad
\includegraphics[height=26mm]{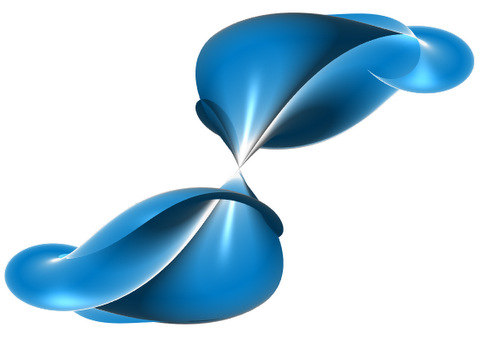} 
\end{array}
$
\end{small}
\caption{Degenerate singularities. Top: $\beta(t)=t$, $A(t)=B(t)=1$. 
Bottom: $\beta(t)=t$, $A(t)=B(t)=-1$. }   \label{degen_example}
\end{figure}   

\subsection{Singularities where the derivative vanishes}  \label{rankzerosection}
Any potential pair $(\hat \chi, \hat \psi)$ with $\chi_1(x_0)= \psi_{-1}(y_0) =0$ produces a 
pseudospherical frontal the derivative of which has rank 0 at $(x_0,y_0)$.
Examples computed with $\hat \chi = \hat \psi =  (-B(t)e_1 \lambda^{-1} - \beta(t)/2 e_3 + A(t) e_1 \lambda) \dd t$,
with $\beta(t)=1$ and $A$ and $B$ vanishing are shown in Figure \ref{rankzero_figure}.

\begin{figure}[ht]
\centering
\begin{small}
$
\begin{array}{cc}
\includegraphics[height=22mm]{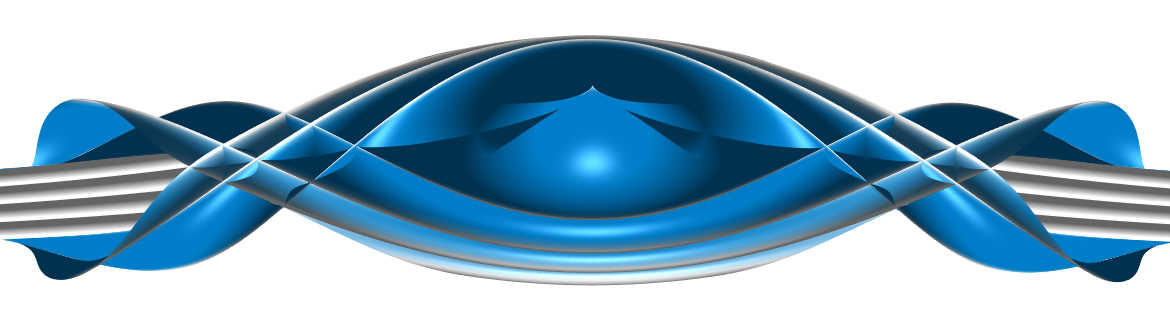}  
 \quad & \quad
\includegraphics[height=22mm]{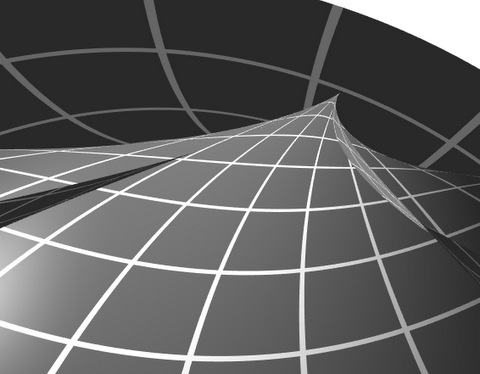}  
\vspace{3ex}\\

\includegraphics[height=26mm]{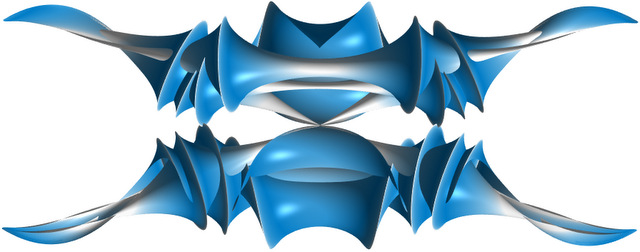}  
 \quad & \quad
\includegraphics[height=26mm]{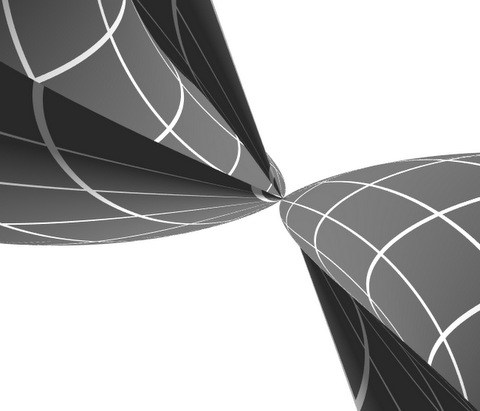}  
\end{array} 
$
\end{small}
\caption{Rank zero singularities. Top: $\beta(t)=1$, $A(t)=B(t)=t$. 
Bottom: $\beta(t)=1$, $A(t)=t$, $B(t)=-t$. (See Section \ref{rankzerosection}).}   \label{rankzero_figure}
\end{figure}

\subsection{Global properties of solutions}
If we consider a surface generated by singular curve data
$(\kappa, \tau)$, 
where $|\kappa(t)| \to \infty$ as $t \to \pm \infty$ and $\tau$ is bounded, then  the solution 
becomes concentrated spatially for large $(u,v)$,
 with a spiral in the $u$ direction and many singularities in the $v$ direction.  This means that 
computing a finite sub-domain gives a realistic sense of what the surface looks like, as in 
Figure \ref{inflectionfig1}.  More examples are shown in Figure \ref{figureunbounded}.
\begin{figure}[ht]
\centering
$
\begin{array}{ccc}
\includegraphics[height=30mm]{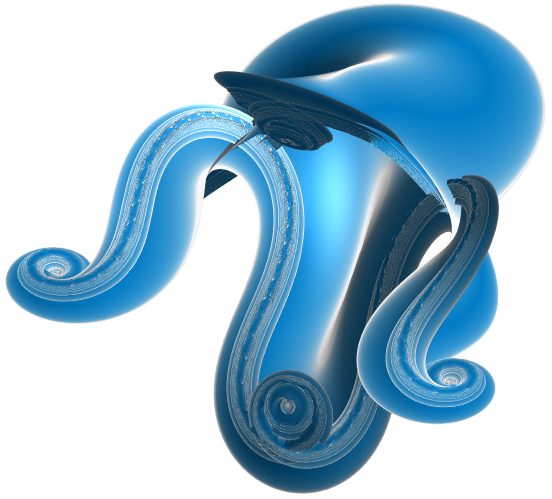}  \,\, &  \, \,
\includegraphics[height=30mm]{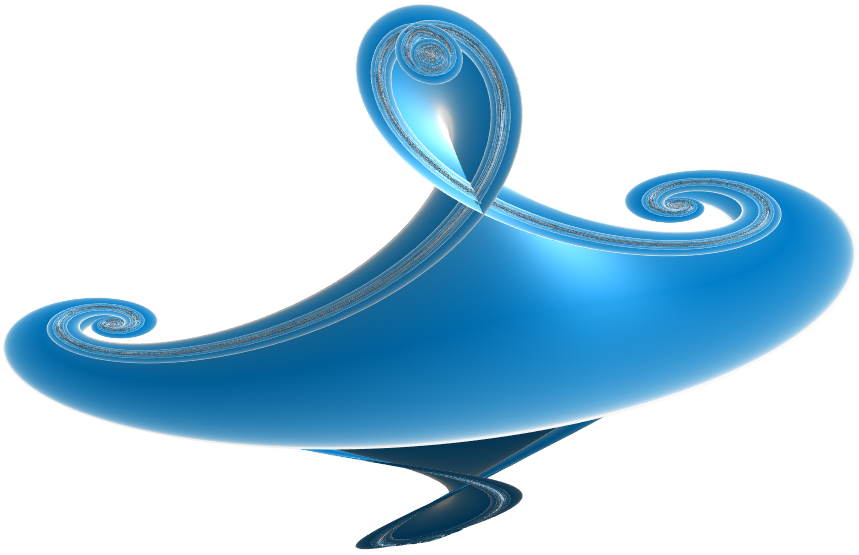} \,\, &  \, \,
\includegraphics[height=30mm]{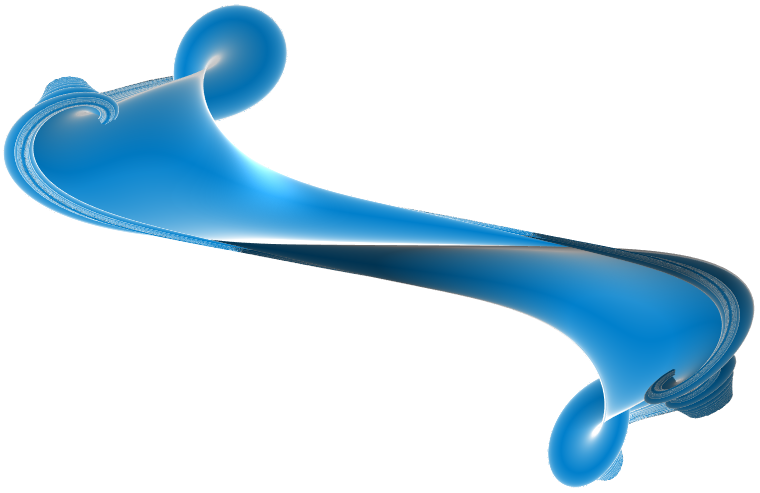}  
\end{array}
$
\caption{Pseudospherical surfaces generated from curves with unbounded curvature functions.
Left:  $\kappa(t) = 2-t^2$, $\tau(t)=0$. Middle: $\kappa(t)=\exp(t^2)$, $\tau(t)=0$. Right:
$\kappa(t)=t^2$ and $\tau(t)=1/2$.   }
 \label{figureunbounded}
\end{figure}

\newpage

\providecommand{\bysame}{\leavevmode\hbox to3em{\hrulefill}\thinspace}
\providecommand{\MR}{\relax\ifhmode\unskip\space\fi MR }
\providecommand{\MRhref}[2]{%
  \href{http://www.ams.org/mathscinet-getitem?mr=#1}{#2}
}
\providecommand{\href}[2]{#2}

\end{document}